\documentclass[10pt]{amsart}
\usepackage{amsmath, amsthm, amssymb}
\usepackage{mathrsfs}

\providecommand{\noopsort[1]{}}
\usepackage{bbm}
\usepackage{ifthen}
\numberwithin{equation}{section}
\usepackage{a4}
\usepackage[active]{srcltx}
\usepackage{verbatim}

\usepackage{hyperref}

\newtheorem{thm}{Theorem}[section]

\newtheorem{prop}[thm]{Proposition}
\newtheorem{lem}[thm]{Lemma}

\theoremstyle{remark}
\newtheorem{rem}[thm]{Remark}

\theoremstyle{definition}

\renewcommand{\Re}{{\rm Re}\,}

\newcommand{\one}{\mathbbm{1}}

\newcommand{\sol}{{\rm{sol}}}
\newcommand{\ot}{\otimes}

\newcommand{\n}{\Vert}

\newcommand{\FF}{\mathbb{F}}

\newcommand{\Vap}{V_{\alpha}^p([0,T]\times\Om;E)}
\newcommand{\Vapo}{V_{\alpha}^p([0,T_0]\times\Om;E)}

\newcommand{\Vaq}{V_{\alpha}^q([0,T]\times\Om;E)}

\newcommand{\form}[3]{\ifthenelse{\equal{#2}{}}{\mbox{$ #1\Big[\, \cdot\,  , \,
\cdot\,  \Big]$}}{
\mbox{$ #1\Big[ #2 , #3 \Big]$}}}
\newcommand{\qform}[2]{\ifthenelse{\equal{#2}{}}{\mbox{$ #1\Big[\cdot \Big]$}}{
\mbox{$ #1\Big[ #2 \Big]$}}}
\newcommand{\ip}[2]{\ifthenelse{\equal{#1}{}}{\mbox{$ \Big( \,\cdot\; \vline \;
\cdot \, \Big) $}}{
\mbox{$ \Big( #1 \;  \vline \; #2 \Big)$}}}
\newcommand{\norm}[1]{\ifthenelse{\equal{#1}{}}{\mbox{$\|\cdot\|$}}{\mbox{$\| #1
\|$}}}
\newcommand{\betr}[1]{\ifthenelse{\equal{#1}{}}{\mbox{$|\cdot|$}}{\mbox{$| #1
|$}}}
\newcommand{\dual}[2]{\ifthenelse{\equal{#1}{}}{\mbox{$ \Big\langle \,\cdot\; ,
\; \cdot \, \Big\rangle $}}{
\mbox{$ \Big\langle #1 \;  , \; #2 \Big\rangle$}}}
\newcommand{\pdual}[2]{\ifthenelse{\equal{#1}{}}{\mbox{$ \Big[ \,\cdot\; , \;
\cdot \, \Big] $}}{
\mbox{$ \Big[ #1 \;  , \; #2 \Big]$}}}
\newcommand{\bdual}[2]{\ifthenelse{\equal{#1}{}}{\mbox{$ \Big\langle \,\cdot\; ,
\; \cdot \, \Big\rangle_* $}}{
\mbox{$ \Big\langle #1 \;  , \; #2 \Big\rangle_*$}}}

\newcommand{\CR}{\mathbb{R}}
\newcommand{\CC}{\mathbb{C}}
\newcommand{\CN}{\mathbb{N}}

\renewcommand{\P}{\mathbb{P}}

\newcommand{\C}{\mathcal{C}}

\newcommand{\bS}{\mathbf{S}}

\newcommand{\cF}{\mathscr{F}}

\newcommand{\expect}{\mathbb{E}}
\newcommand{\cL}{\Lambda}
\newcommand{\half}{\frac{1}{2}}

\newcommand{\ghe}{\gamma (H,E)}
\newcommand{\inject}{\hookrightarrow}

\def\R{{\mathbb R}}
\def\C{{\mathbb C}}

\newcommand{\E}{{\mathbb E}}
\renewcommand{\P}{{\mathbb P}}
\newcommand{\F}{{\mathscr F}}
\renewcommand{\H}{\mathscr{H}}

\newcommand{\OO}{\mathcal{O}}
\renewcommand{\a}{\alpha}
\renewcommand{\b}{\beta}

\newcommand{\la}{\lambda}
\newcommand{\om}{\omega}
\renewcommand{\O}{\Omega}
\newcommand{\Om}{\Omega}

\newcommand{\calL}{{\mathscr L}}

\newcommand{\g}{\gamma}
\newcommand{\embed}{\hookrightarrow}
\newcommand{\s}{^*}
\newcommand{\lb}{\langle}
\newcommand{\rb}{\rangle}

\newcommand{\limn}{\lim_{n\to\infty}}

\newcommand{\Dom}{\mathsf{D}}

\begin{document}

\title[Approximating the coefficients in semilinear SPDEs]
{Approximating the coefficients in semilinear stochastic partial differential equations}
\author{Markus Kunze}
\author{Jan van Neerven}
\address{Delft Institute of Applied Mathematics, Delft University of Technology,
P.O. Box 5031, 2600 GA Delft, 
The Netherlands}
\email{M.C.Kunze@TUDelft.nl, J.M.A.M.vanNeerven@TUDelft.nl}
\date{\today}
 
\begin{abstract} 
We investigate, in the setting of UMD Banach spaces $E$, the continuous
dependence on the data $A$, $F$, $G$ and $\xi$ of mild solutions of semilinear stochastic evolution equations
with multiplicative noise of the form
$$
             \left\{
\begin{aligned}
           dX(t) & =  [AX(t) + F(t,X(t))] \, dt + G(t,X(t)) \, dW_H(t),\quad
t\in
	   [0,T],\\
	    X(0) & =  \xi,
\end{aligned}\right.
$$
where $W_H$ is a cylindrical Brownian motion in a Hilbert space $H$.
We prove continuous dependence of the compensated solutions $X(t) - e^{tA}\xi$
in the norms 
$L^p(\O;C^\la([0,T];E))$ 
assuming that the approximating operators $A_n$ are uniformly sectorial and
converge to $A$ in the strong
resolvent sense, and that the approximating nonlinearities $F_n$ and $G_n$ are uniformly Lipschitz
continuous in suitable 
norms and converge to $F$ and $G$ pointwise. 
Our results are applied to a class of semilinear parabolic SPDEs 
with finite-dimensional multiplicative noise.
\end{abstract}

\thanks{The authors are supported by VICI subsidy 639.033.604
of the Netherlands Organisation for Scientific Research (NWO)}
\maketitle

\section{Introduction}

We consider semilinear stochastic evolution equations with multiplicative noise
of the form
\begin{equation}\label{SDE}\tag{SCP} 
             \left\{
\begin{aligned}
           dX(t) & =  [AX(t) + F(t,X(t))] \, dt + G(t,X(t)) \, dW_H(t),\quad
	   t\in[0,T],\\
	    X(0) & =  \xi,
\end{aligned}\right.
\end{equation}
where $A$ is the generator of a strongly continuous analytic semigroup $\bS = (S(t))_{t\ge 0}$ 
on a UMD
Banach space
$E$, the driving process $W_H$ is a cylindrical 
Brownian motion in a Hilbert space $H$  
defined on some probability space $\Om$ (see Section \ref{sect.solution} 
for the definition), the functions $F: [0,T]\times\Om\times E \to E$ and $G :
[0,T]\times\Om\times E \to \mathscr{L}(H,E)$
satisfy suitable measurability and Lipschitz continuity conditions,  
and the initial value $\xi$ is an $E$-valued random variable on $\Om$.

The theory of stochastic integration in UMD Banach spaces yields existence,
uniqueness and regularity of mild solutions 
\cite{vNVW08, vNVW07}.
It is a natural question how this solution depends on the `coefficients' $A,
F$, $G$
and the initial datum $\xi$. 
Our main abstract results are Theorems \ref{t.coeffdep} and \ref{t.hoelderconv}
which assert, roughly speaking, that the solution $X(\cdot)$ and the compensated 
solution $X(\cdot )-S(\cdot)\xi$ depend continuously on $A$, $F$, $G$ and $\xi$ 
simultaneously with respect to the norms of $L^p(\Omega;C([0,T];E))$ 
and $L^p(\Omega;C^\lambda([0,T];E))$, 
respectively.

In the case when $E$ is a Hilbert space, concerning dependence of the solution
on the initial datum $\xi$ we refer 
to Da Prato and Zabczyk
\cite{dprzab2}; see also the recent work by Marinelli, Pr\'ev\^ot,
and R\"ockner \cite{mpr08} for the 
case of Poisson noise. Approximations of the functions $F$ and $G$ are
considered in Peszat and Zabczyk 
\cite{pezab95} and Seidler \cite{seidler97}. Under more restrictive assumptions than ours,
simultaneous approximations of 
$A$, $F$, $G$, and $\xi$ were considered by Brze\'zniak \cite{Brz97} in the
setting of UMD Banach spaces with type $2$. 

These approximation results provide a justification for the use of numerical
schemes, where necessarily one 
replaces continuous objects by discretized approximations. Furthermore,
approximating $A$ by bounded operators 
$A_n$ (such as their Yosida approximations, see Section
\ref{subsec.Yosida}) is often helpful on a technical level, 
for instance in the standard proofs of the It\^o lemma in infinite dimensions
\cite{BrzElw, dprzab2}. 

We apply our abstract results to the stochastic partial
differential equation (SPDE)
$$
\left\{
\begin{aligned}\displaystyle
\frac{\partial u}{\partial t}(t,x ) &= \mathcal{A}u (t,x) + f(u(t,x)) +
\sum_{k=1}^K g_k(u(t,x))\frac{\partial W_k}{\partial t} (t), 
&& x \in \mathcal{O}, && t\ge 0,
\\
 u(t,x) &= 0, 
&& x \in \partial \mathcal{O}, &&  t\ge 0, 
\\
 u(0,x) &= \xi(x), && x \in \mathcal{O} .&&
\end{aligned}\right.
$$
Here $\mathcal{O}$ is a bounded open domain in $\CR^d$ and
$$\mathcal{A}u(x)= \sum_{i=1}^d \frac{\partial}{\partial
x_i}\Big(a_{ij}(x)\sum_{j=1}^d\frac{\partial u}{\partial x_j}(x)\Big)
+ \sum_{j=1}^d b_j(x) \frac{\partial u}{\partial x_j}(x)
$$
is a second order differential operator in divergence form  
whose coefficients ${\bf a} = (a_{ij})$ and ${\bf b} = (b_j)$ satisfy suitable
boundedness and uniform 
ellipticity conditions.
The functions $f$ and $g_k$ are Lipschitz continuous and the driving processes
$W_k$ are independent 
real-valued standard Brownian motions.
For this problem our abstract results imply the following approximation result.
We let $\n f\n_{\rm Lip} = \sup_{t\not=s} \frac{|f(t)-f(s)|}{|t-s|}$
denote the Lipschitz seminorm of a function $f$.

\begin{thm}\label{thm:main} Let $\mathbf{a}, \mathbf{a}_n \in L^\infty
(\mathcal{O}; \CR^{d\times d})$,  let
$\mathbf{b}, \mathbf{b}_n \in L^\infty (\mathcal{O}; \CR^d)$, and let $f$,
$f_n$, $g_k$, $g_{k,n}:\R\to \R$ be 
Lipschitz continuous. Assume that there exist finite constants $\kappa,C>0$ such
that:
\begin{enumerate}
\item[\rm(i)] $\mathbf{a},\mathbf{a}_n$ are symmetric and
$\mathbf{a}{\bf x}\cdot {\bf x} , \ \mathbf{a}_n{\bf x} \cdot {\bf x} \geq
\kappa |{\bf x} |^2$ for all ${\bf x} \in \CR^d$;
\item[\rm(ii)]
$\norm{\mathbf{a}}_\infty, \, \norm{\mathbf{a}_n}_\infty, \,
\norm{\mathbf{b}}_\infty, \, \norm{\mathbf{b}_n}_\infty \leq C$;
\item[\rm(iii)]
$\n f\n_{\rm Lip},\, \n f_n\n_{\rm Lip},\, \n g_k\n_{\rm Lip},\, \n
g_{k,n}\n_{\rm Lip}\le C.$
\end{enumerate}
Assume further that
\begin{enumerate}
\item[\rm(iv)] $\limn \mathbf{a}_n = \mathbf{a},\, \lim_{n\to \infty} 
\mathbf{b}_n = \mathbf{b}$ almost everywhere on $\OO$; 
\item[\rm(v)] $\limn f_n = f,\, \limn g_{k,n} = g_k$ pointwise on $\OO$.
\end{enumerate}
Let $1<p<\infty$. If $\xi_n\to \xi$ in $L^p(\OO)$, 
the approximate mild solutions $u_n$ converge to the mild solution $u$ in the following
compensated sense:
for all $1\le q<\infty$ and $0\le \lambda < \frac12$   
we have
\[  u_n - S_n(\cdot)\xi_n \to  u-S(\cdot)\xi \ \hbox{in} \ 
L^q(\Omega;C^\lambda ([0,T]; L^p(\OO))). \]
Here
$S_n(\cdot)$ and $S(\cdot)$ denote the strongly continuous analytic semigroups
generated by the $L^p(\OO)$-realizations of $\mathcal{A}_n$ and $\mathcal{A}$.
Furthermore, for all $1\le q<\infty$ we have $ u_n\to u \ \hbox{in} \ L^q(\Omega;C([0,T]; L^p(\OO)))$.
\end{thm}

A slightly more general version of this theorem allowing for random initial conditions
is presented below (Theorem \ref{thm:5.3}).

It is possible to extend our results to SPDEs with locally Lipschitz continuous
nonlinearities, measurable initial values, and infinite-dimensional noise; also regularity in both space and time can be accounted for. These extensions involve
the use of interpolation techniques and require additional assumptions on the 
domains $\Dom(A)$ and $\Dom(A_n)$. In order to keep this article at a reasonable 
length we have chosen to postpone these extensions to a forthcoming publication \cite{forthcoming}. 

The organization of the paper is as follows. 
In Section \ref{sect.preliminaries} we prove an abstract approximation result for certain spaces of $\gamma$-radonifying operators. After recalling some 
results about solving the abstract problem 
\eqref{SDE} in Section \ref{sect.solution}, we prove our main abstract approximation results in Section \ref{sect.dependence}. 
Theorem \ref{thm:main} is proved in Section \ref{sect.applications}, where some
further applications are presented as well.

Throughout this article, all vector spaces are real. Whenever this is 
needed, e.g. when using spectral theory, we shall pass to their complexifications.
We assume the reader to be familiar with standard Banach space concepts such as
the UMD property and the notions of type and cotype. 
For more information we recommend the survey articles by Burkholder and Maurey in the Handbook of Geometry of Banach Spaces  \cite{handbook1,handbook2}. 

When $P_n(\phi)$ and $Q_n(\phi)$ are certain quantities depending on an index $n$ and a function $\phi$, we use the notation 
$P_n(\phi)\lesssim Q_n(\phi)$ to indicate that there is a constant $C$, independent of 
$\phi$, such that $P_n(\phi)\le
CQ_n(\phi)$ holds for all indices $n$.  
Unless otherwise stated this constant is allowed to depend on all other relevant
data. We write $P_n(\phi)\eqsim Q_n(\phi)$ if $P_n(\phi)\lesssim Q_n(\phi)$  
and $Q_n(\phi)\lesssim P_n(\phi)$.

\section{Approximation of $\gamma$-radonifying
operators}\label{sect.preliminaries}

We begin with a brief discussion of spaces of $\g$-radonifying operators, which
play an important role in the theory of stochastic integration in UMD Banach
spaces.

Let $\mathscr{H}$ be a Hilbert space (below we shall take $\mathscr{H} =
L^2(0,T;H)$, where $H$ is another Hilbert 
space) and $E$ be a Banach space.
Any finite rank operator $R:\mathscr{H}\to E$ can be represented in the form
$\sum_{n=1}^N h_n\otimes x_n$, where the vectors $h_n$ are orthonormal in
$\mathscr{H}$ and the vectors $x_n$ 
belong to $E$. For such an operator we define
\[ \norm{R}_{\gamma (\mathscr{H},E)}^2 := \expect\Big\|
\sum_{n=1}^N\gamma_nRh_n\Big\|^2.\]
Here, and in what follows, $(\gamma_n)_{n=1}^N$ is a
sequence of independent real-valued standard Gaussian random variables. 

It is easy to check that the above identity defines a norm on the space
$\H\otimes E$ of all finite rank 
operators from $\mathscr{H}$ to $E$. The completion of  $\H\otimes E$ with
respect to this norm is 
denoted by $\gamma(\mathscr{H},E)$. This space is contractively embedded into
$\mathscr{L}(\mathscr{H},E)$. 
A bounded operator in $\calL(\H,E)$ is called {\em $\gamma$-radonifying} if it
belongs to $\g(\mathscr{H},E)$. 

For all $R\in \gamma(\H,E)$ we have the identity
$$ \n R \n_{\g(\H,E)}^2 = \sup_h \E \Big\n \sum_{j=1}^k \g_j R h_j \Big\n^2,$$
where the supremum is taken over all finite orthonormal systems 
$h = \{h_j\}_{j=1}^k$ in $\H$. 
A bounded operator $R$ from $\H$ to $E$ is called {\em $\g$-summing} if
the above supremum is finite. This supremum, denoted by
$\n R\n_{\gamma_\infty(\H,E)}$, turns the space of all $\g$-summing operators
from $\H$ to $E$ into a Banach space. By definition we have an isometric 
inclusion $\gamma(\H,E)\subseteq \gamma_\infty(\H,E)$.
It follows from a result of Hoffmann-J{\o}rgensen and Kwapie\'n \cite{HofJor, Kwa}
that $\g(\H,E)=\g_\infty(\H,E)$ if (and only if) $E$ does not contain a 
closed subspace isomorphic to $c_0$.

The space $\gamma (\mathscr{H},E)$ enjoys
the following {\it ideal property}: if
$T \in \mathscr{L}(\mathscr{H}_2,
\mathscr{H}_1)$, $S\in \mathscr{L}(E_1,E_2)$ and $R\in \gamma
(\mathscr{H}_1,E_1)$, where $\mathscr{H}_1,
\mathscr{H}_2$ are Hilbert spaces and $E_1,E_2$ are Banach spaces, 
then $SRT \in \gamma (\mathscr{H}_2,E_2)$ and 
\[ \norm{SRT}_{\gamma (\mathscr{H}_2,E_2)}\leq
\norm{S}_{\mathscr{L}(E_1,E_2)}\norm{R}_{\gamma (\mathscr{H}_1,E_1)}
 \norm{T}_{\mathscr{L}(\mathscr{H}_2,\mathscr{H}_1)}.
\]
The analogous result holds for the space $\gamma_\infty(\H,E)$.

For more information and proofs we refer to the review article \cite{vNsurvey}
and the references given therein.

We are mainly interested in the case $\H = L^2(X,\mu;H)$, where $\mu$ is a 
$\sigma$-finite measure on some measurable space $X$, and $H$ is another 
Hilbert space. In this situation we say that a function 
$\Phi:X\mapsto \calL(H,E)$ {\em represents} a bounded
operator 
$R: L^2(X,\mu;H)\to E$ if the following two conditions are satisfied:
\begin{enumerate}
\item[\rm(i)] For all $x\s\in E\s$ the function $t\mapsto \Phi\s(t) x\s$ belongs
to $L^2(X,\mu;H)$; 
\item[\rm(ii)] For all $f\in L^2(X,\mu;H)$ and $x\s\in E\s$ we have
$$\lb R f,x\s\rb = \int_X [f(t), \Phi\s(t)x\s]\,d\mu(t).$$
\end{enumerate}
If $t\mapsto \Phi(t)h$ is strongly measurable for all $h\in H$, then the
operator $R$ is uniquely determined by $\Phi$; see \cite{vNVW07}.

It will be important to have criteria for checking whether a given function
$\Phi:X\to\calL(H,E)$ represents
an operator in the space $\g (L^2(X,\mu;H),E)$. We begin
with the following simple result; see \cite{DNW}.

\begin{prop}\label{l.tensor} 
For all $f\in L^2(X;\mu)$ and $S \in \g (H,E)$ the function $f\otimes S: t\mapsto f(t)S$
represents a unique operator 
 $R_{f\otimes S}\in \g(L^2(X,\mu;H),E)$, which given by
$$ R_{f\otimes S} g = \int_X f(t) Sg(t)\,d\mu(t),$$ and we have 
$$ \big\n R_{f\otimes S}\big\n_{\g(L^2(X,\mu;H),E)} = 
\n f\n_{L^2(X,\mu)}\n S\n_{\ghe}.$$
\end{prop}

The following sufficient condition for a function $\Phi:(a,b)\to\calL(H,E)$ to represent an
element of $\g (L^2(a,b;H),E)$ is
a simple extension of a result due to Kalton and Weis \cite{kwpre}. For the
proof we refer to
\cite[Proposition 13.9]{vNsurvey}.

\begin{prop}\label{p.gammaest}
Let $\Phi : (a,b) \to \ghe$ be continuously differentiable with $$\int_a^b
(t-a)^{\half}
\norm{\Phi '(t)}_{\ghe}\, dt < \infty.$$ 
Then $\Phi $ represents a unique operator $R_\Phi\in \gamma (L^2(a,b;H),E)$ and
\[ \norm{R_\Phi }_{\gamma (L^2(a,b;H),E)} \leq
(b-a)^\half\norm{\Phi(b-)}_{\ghe}+\int_a^b(t-a)^\half
 \norm{\Phi'(t)}_{\ghe}\, dt . 
\]
\end{prop}

A subset $\mathscr{T}\subseteq \mathscr{L}(E,F)$, where $E,F$ are Banach spaces,
is called {\it $\gamma$-bounded}, if there exists $C\geq 0$ such that for all
finite sequences
$x_1,\dots,x_N \in E$ and $T_1,\dots,T_N \in \mathscr{T}$ we have
\[ \expect\Big\|\sum_{n=1}^N\gamma_nT_nx_n\Big\|^2\leq C^2\expect\Big\|
 \sum_{n=1}^N \gamma_nx_n\Big\|^2.
\]
The infimum over all admissible constants $C$ is called the {\it $\gamma$-bound}
of $\mathscr{T}$
and is denoted by $\gamma (\mathscr{T})$.
For more information on $\g$-boundedness and the related notion of
$R$-boundedness we refer to \cite{CPSW} and the lecture notes \cite{kunweis}.
We will need the following elementary fact.

\begin{prop}\label{prop.g-bound-cov} If $\mathscr{T}$ is $\g$-bounded, then the
closure in the strong operator topology of its absolute convex hull is
$\g$-bounded as well, and
$\g(\overline{\rm co}(\mathscr{T})) = \g(\mathscr{T})$.
\end{prop}

We will also need the following sufficient condition for $\g$-boundedness due to
Weis \cite{weis01}.

\begin{prop}\label{p.integrable-derivative}
Suppose $\Phi:(a,b)\to \calL(H,E)$ is continuously differentiable. If
$\Phi'$ is integrable, then the family $\mathscr{T}_\Phi:= \{\Phi(t):\
t\in(a,b)\}$ is $\g$-bounded and
$$\g(\mathscr{T}_\Phi)\le  \n \Phi(a+) \n + \int_a^b \n \Phi'(t)\n\,dt.$$
\end{prop}

In order to be able to state a second sufficient condition for being a member of
$\g(L^2(a,b;H),E)$ 
we need to make a simple preliminary observation.
Let $\Phi:(a,b)\to \calL(H,E)$ be a function of the form
$$ \Phi = \sum_{m=1}^M f_m\ot (h_m\ot x_m),$$
with $f_m\in L^2(a,b)$, $h_m\in H$, $x_m\in E$. Then $\Phi$ represents the
operator 
$$R_\Phi =\sum_{m=1}^M (f_m\ot h_m)\ot x_m$$
which is of finite rank and therefore belongs to $\g(L^2(a,b;H),E)$.
It will be important later on that the linear span of all such operators is
dense in $\g(L^2(a,b;H),E)$; see \cite{vNW05}.

Now let $F$ be a second Banach space and suppose that $M:(a,b)\to\calL(E,F)$ is
a function with the property that 
$t\mapsto M(t)x$ is strongly measurable and bounded for all $x\in E$.
If $\Phi$ is as above, 
then the function $M\Phi: t\mapsto M(t)\Phi(t)$ is strongly measurable and
represents a unique bounded 
operator $R_{M\Phi}$ from $L^2(a,b;H)$ to $F$. 

Under these assumptions one has the following result, also due to Kalton and Weis 
\cite{kwpre}; a proof can be found in \cite{vNsurvey}.

\begin{prop}\label{p.gbdd} Let $E$ and $F$ be Banach spaces and 
suppose $\Phi$ represents an operator $R_\Phi \in \g (L^2(a,b;H),E)$.
If $M$ has $\g$-bounded range $\mathscr{M}=\{M(t): \ t\in (a,b)\}$, 
then $R_{M\Phi} \in \g_\infty(L^2(a,b;H),F)$ and
$$ \n R_{M\Phi}\n_{\g_\infty(L^2(a,b;H),F)} \le \gamma(\mathscr{M})
\n R_{\Phi}\n_{\g(L^2(a,b;H),E)}.
$$
Consequently, the mapping
$R_\Phi\mapsto R_{M\Phi}$ has a unique extension
to a bounded linear operator (also denoted by M) from $\gamma (L^2(a,b;H),E)$ to
$\gamma_\infty (L^2(a,b;H),F)$ 
of norm at most $\gamma(\mathscr{M})$.
\end{prop}

In our main results we shall always assume that $E=F$ is a UMD space.
Such spaces, being reflexive, cannot contain isomorphic copies of $c_0$, and therefore 
$\gamma_\infty(L^2(a,b;H),F) = \gamma(L^2(a,b;H),F)$ in this situation.

After these preparations we are in a position to state the main approximation 
lemma of this section.

\begin{lem}\label{l.gammaconv}  Let $E$ and $F$ be Banach spaces. 
Let the functions 
$M_n, M: (a,b)\to \mathscr{L}(E, F)$ satisfy the following conditions:
\begin{enumerate}
\item For all $x \in E$ the functions $M_n(\cdot )x$ and $M(\cdot )x$ are
continuously differentiable on $(a,b)$;
\item For all $x\in E$ we have $\limn M_n(t)x  =  M(t)x$ 
and $ \limn M_n'(t)x  =  M'(t)x$ uniformly on compact subsets of $(a,b)$;
\item The sets $\mathscr{M}_n = \{M_n(t): \ t\in (a,b)\}$ are $\gamma$-bounded
and $\sup_n \gamma(\mathscr{M}_n) <\infty.$
\end{enumerate}
Then for all $R\in \gamma(L^2(a,b;H),E)$ we have
$M_nR, MR \in \gamma(L^2(a,b;H),F)$ and 
\[\limn M_nR =  MR\ \  \mbox{in} \ \, \gamma (L^2(a,b;H),F) . \]
Here the operators $M_n,M: \gamma(L^2(a,b;H),E)\to \gamma_\infty(L^2(a,b;H),F)$ 
are as in Proposition \ref{p.gbdd}. 
\end{lem}

\begin{proof}
First we consider the case where $R$ is represented by the function 
$\one_{(a', b')}\otimes S$, where $a<a'<b'<b$ and 
$S\in \gamma (H,E)$ is a fixed finite rank
operator, say $S = \sum_{j=1}^k h_j\otimes x_j$ with the vectors $h_j\in H$ orthonormal.

Proposition \ref{p.gammaest} (with $a,b$ replaced with $a',b'$) 
implies that $M_n(\cdot)R$ and $M(\cdot)R$ belong to $\gamma (L^2(a',b';H),F)$, and hence to
$\gamma (L^2(a,b;H),F)$,
and by Propositions \ref{p.gbdd} and \ref{l.tensor},
\[\norm{ M_n(\cdot)R}_{\gamma (L^2(a,b;H),F)}  \leq C\norm{
\one_{(a',b')}}_{L^2(a,b)} 
\norm{S}_{\gamma (H,E)} \le C(b-a)^\frac12 \norm{S}_{\gamma (H,E)}
\] 
with $C := \sup_{n} \gamma(\mathscr{M}_n)$.
Taking strong limits and invoking Proposition \ref{prop.g-bound-cov},
$\{M(t): \ t\in (a,b)\}$ is $\gamma$-bounded with $\gamma$-bound 
at most $C$, and therefore the same estimates hold with $M_n$ replaced by $M$.

We claim that 
\[\limn M_n(\cdot)R=  M(\cdot)R \quad \mbox{in}\,\,  \gamma
(L^2(a,b;H),F) . 
\]
From the representation $S = \sum_{j=1}^k h_j\otimes x_j$ it follows that
\begin{align*} 
\ & \norm{r^{-1}\big[M_n(t+r)S - M_n(t)S\big] - M_n'(t)S}_{\gamma (H,F)}\\
& \qquad =  \Big(\expect \Big\| \sum_{j=1}^N \gamma_j \big(r^{-1}\big[M_n(t+r) -
M_n(t)\big] - M_n'(t)\big)S h_j 
\Big\|_{F}^2\Big)^\half\\
& \qquad \leq  \sum_{j=1}^N \big\n \big(r^{-1}\big[M_n(t+r) - M_n(t)\big] -
M_n'(t)\big)S h_j\big\n_{F} \to 0 
\ \ \hbox{as} \ r\to 0,
\end{align*}
for all $t \in (a,b)$ and $n\geq 1$. Hence $M_n(\cdot)S$ is differentiable on $(a, b)$ as
a $\gamma (H,F)$-valued function with derivative $M_n'(\cdot)S$. Similarly we see
that $M(\cdot )S$ is differentiable with derivative $M'(\cdot)S$, and arguing as 
above we see that $M_n'(\cdot)S$ and $M'(\cdot)S$ are continuous on $(a,b)$.
It now follows from Proposition \ref{p.gammaest} that
\begin{equation}\label{eq.domest}
 \begin{aligned} 
\norm{ M_n R - MR}_{\gamma (L^2(a,b;H),F)} & = \norm{M_nS - MS}_{\gamma (L^2(a',b';H),E)}\\
& \leq  \int_{a'}^{b'} (t-a')^\half \norm{ M_n'(t)S -
M'(t)S}_{\gamma (H,F)}\,dt\\
& \qquad  
+ (b'-a')^\half \norm{ M_n(b')S - M(b')S}_{\gamma (H,F)} ,
 \end{aligned}
\end{equation}
where the integral is finite since $M_n'(\cdot)S$ and $M'(\cdot)S$ are
continuous $\gamma(H,F)$-valued functions on $[a',b']$. 
Since $M_n'(t) \to M'(t)$ strongly for every $t \in (a,b)$ we see, using that
$S$ has finite rank as before, that $M_n'(t)S \to M'(t)S$ in $\gamma (H,F)$. 
This convergence is uniform on $[a' , b']$, 
and therefore the integral in the estimate \eqref{eq.domest} converges to 0. 
Convergence of the second term in \eqref{eq.domest} follows from  
$M_n(b') \to M(b')$ strongly and the fact that $S$ has finite
rank. This proves the claim.

Since the multiplication operators associated with $M_n$ are uniformly bounded by
Proposition \ref{p.gbdd} and assumption (3), the general case follows from a density argument.
To that end, observe that the step functions with values in the finite
rank operators and support in a proper two-sided subinterval $(a',b')$ of $(a,b)$ 
are dense in $\gamma (L^2(a,b;H),E)$. This follows, e.g.,  from \cite[Proposition 2.4]{vNVW07}.
\end{proof}

From now on, we will no longer distinguish between a function
$\Phi: (a,b) \to \mathscr{L}(H,E)$ and the operator $R_\Phi : L^2(a,b;H) \to E$
represented by it.

\section{Semilinear stochastic evolution equations}\label{sect.solution}

In this section we collect some known facts concerning the existence and
uniqueness of mild solutions of 
the problem \eqref{SDE},
\begin{equation}\label{SDE2}\tag{SCP} \left\{
\begin{aligned}
           dX(t) & =  [AX(t) + F(t,X(t))] \, dt + G(t,X(t)) \, dW_H(t),\quad t\in [0,T],\\
	    X(0) & =  \xi .
\end{aligned}\right. 
\end{equation}

The probability space
$(\Omega, \Sigma, \P)$, endowed with a filtration $\FF = (\cF_t)_{t\geq 0}$, 
is fixed throughout this paper. 
The driving process
$W_H: L^2(\R_+;H)\to L^2(\O)$ is an $\FF$-cylindrical Brownian motion, that is,
it is a bounded linear operator from $L^2(\R_+;H)$ to $L^2(\O)$ 
with the following properties:
\begin{enumerate}
\item[(i)] for all $f\in L^2(\R_+;H)$ the random variable
$W_H(f)$ is centred Gaussian
\item[(ii)] for all $t\in \R_+$ and $f\in L^2(\R_+;H)$ with support in $[0, t]$, $W_H(f)$ is $\F_t$-measurable.
\item[(iii)] for all $t\in \R_+$ and $f\in L^2(\R_+;H)$ with support in $[t, \infty)$, $W_H(f)$ is independent of $\F_t$.
\item[(iv)] for all $f_1,f_2\in L^2(\R_+;H)$ we have
$ \E (W_H(f_1)\cdot W_H(f_2)) = [f_1,f_2]_{L^2(\R_+;H)}.$
\end{enumerate}
It is easy to see that for all $h\in H$ the process $(W_H(t)h )_{t\ge 0}$ defined by
$$W_H(t)h := W_H(\one_{(0,t]}\otimes h)$$
is an $\F$-Brownian motion (which is standard if $\n h\n=1$). Moreover, two such
Brownian motions
$((W_H(t)h_1)_{t\ge 0}$  and $((W_H(t)h_2)_{t\ge 0}$ are independent if and only
if $h_1$ and $h_2$ are orthogonal in $H$. We refer to \cite{vNsurvey}
for a further discussion.

The linear operator 
$A$ is assumed to be closed and
densely on $E$,  
and the functions $F: [0,T]\times \Omega \times E\to E$ and $G: [0,T]\times
\Omega \times E\to \mathscr{L}(H,E)$ 
are strongly measurable and adapted and satisfy suitable Lipschitz and growth 
conditions specified below. 
 
Concerning the operator $A$, we make the following assumption:

\begin{enumerate}
\item[\rm(A)] The operator $A$ generates a strongly continuous analytic
semigroup $\bS=(S(t))_{t\ge 0}$ on $E$.
\end{enumerate}
Recall that a closed operator $A$ generates a strongly continuous analytic
semigroup 
on a Banach space $E$ if and only if $A$ is densely defined and sectorial, i.e.
there exist $M\ge 1$ and $w\in\R$ such that  $\{\la\in\C: \ \Re\la>w\}$ is
contained 
in the resolvent set $\varrho(A)$ and 
$$\sup_{{\rm Re}\la>w} \n (\la - w) R(\la,A)\n \le M.$$
The constants $M$ and $w$ are called the {\em sectoriality constants} of $A$;
in this context we say that $A$ is {\em sectorial of type $(M,w)$}.

If (A) holds, then $S(t)$ maps $E$ into 
the domain $\Dom(A)$ and 
$ \limsup_{t\downarrow 0} t\n AS(t)\n <\infty$. By Proposition
\ref{p.integrable-derivative} this implies 
the following useful fact (see, e.g., \cite[Lemma 4.1]{vNVW08}):

\begin{lem}\label{l.S}
 If $A$ generates a strongly continuous analytic semigroup on a Banach space
$E$, then for all $t\in [0,T]$ 
and $\a>0$ the set $\mathscr{T}_{\alpha,t} := \{s^\a S(s):\ s\in [0,t]\}$ is
$\g$-bounded and
$\g(\mathscr{T}_{\alpha,t})\le Ct^\a$, where $C$ depends on $A$ only through its
sectoriality constants.
 \end{lem} 

Concerning $F$ and $G$ we shall assume:

\begin{itemize}
 \item[{\rm (F)}] The function $F: [0,T]\times \Omega \times E \to E$ is Lipschitz continuous 
and of linear growth in its third variable, 
uniformly in $[0,T]\times \Omega$,
i.e. there exist constants
$L_F$ and $C_F$ such that for all $t \in [0,T], \omega \in \Omega$ and $x,y \in
E$ we have
\[
 \begin{aligned}
  \big\| F(t,\omega, x) - F(t,\omega, y)\big\| & \leq L_F\n x-y\n\\
\big\| F(t,\omega, x )\big\| & \leq C_F(1+\n x\n).
 \end{aligned}
\]
Furthermore, for all $x \in E$ the map $(t,\omega) \mapsto F(t,\omega, x)$ is
strongly measurable and adapted.

 \item[{\rm (G)}] The function $G: [0,T]\times \Omega \times E \to \mathscr{L}
(H,E)$ is $\gamma$-Lipschitz continuous
and of linear growth, uniformly in $[0,T]\times \Omega$, i.e. there exist constants $L_G$
and $C_G$ such that for all finite Borel measures $\mu$ on $[0,T]$, for all $\omega \in \Omega$ and
for all $\phi_1, \phi_2,\phi \in L^2((0,T),\mu; E)\cap \g (L^2((0,T), \mu ),E)=:
L^2_\gamma ((0,T),\mu; E)$ we have
\[
 \begin{aligned}
{} \qquad  \big\n G(\cdot, \omega, \phi_1 ) - G(\cdot, \omega, \phi_2)\big\n_{\gamma
(L^2((0,T),\mu;H),E)} 
& \leq L_G\n \phi_1 -\phi_2\n_{L^2_\g ((0,T),\mu;E)}\\
\big\n G(\cdot, \omega, \phi )\big\n_{\gamma (L^2((0,T),\mu;H),E)} & \leq C_G
\big(1+ \n\phi\n_{L^2_\g ((0,T),\mu;E)}\big).
 \end{aligned}
\]
Furthermore, for all $x \in E$ and $h \in H$ the map $(t,\omega ) \mapsto
G(t,\omega,x)h$ is strongly measurable and adapted.
\end{itemize}

The notion of $\g$-Lipschitz continuity has been introduced in \cite{vNVW08},
where various characterizations and examples were given. In particular,
if $E$ is a type $2$ Banach space (e.g.\ an $L^p$-space with $2\le p<\infty)$, 
then every Lipschitz continuous function with values in $\ghe$ is $\g$-Lipschitz continuous. 

\begin{rem}
It is implicit in condition (G) that for all $\omega\in \Omega$ the functions
$t\mapsto G(t, \omega, \phi(t) )$ should represent an operator from
$L^2((0,T),\mu;H)$ to $E$. Note that the strong measurability in $H$
of $t\mapsto G\s(t, \omega, \phi(t) )x\s$ can be assumed without loss
of generality. Indeed, the weak measurability of this functions 
is clear from the assumptions and, as explained in \cite{vNsurvey},
there is no loss of generality in assuming that $H$ be separable; strong 
measurability then follows from the Pettis measurability theorem.
\end{rem}

\begin{rem}
 In the present context, where the driving process is a 
cylindrical Brownian motion,  it is not necessary to assume completeness of the filtration
and/or progressive measurability of $F$ and $G$ (cf. \cite[Proposition 2.10]{vNVW07}).  
\end{rem}

\begin{rem}
Taking $\phi = \one\otimes x$, we see that condition (G) implies that $t\mapsto G(t,\omega,x)$
defines an element in 
$\gamma (L^2((0,T),\mu;H),E)$ for all $\omega\in\Omega$ and $x\in E$.
\end{rem}

A {\it mild solution} of the problem \eqref{SDE2} is a continuous adapted
process
$X:[0,T]\times \Om \to E$ such that 
\begin{enumerate}
 \item for all $t \in [0,T]$, $s \mapsto S(t-s)F(s, X(s))$ is strongly measurable and belongs to
$L^1((0,t);E)$ almost surely;
 \item for all $t \in [0,T]$, $s \mapsto S(t-s)G(s,X(s))$ is strongly
measurable and stochastically integrable with respect to $W_H$;
 \item for all $t \in [0,T]$ we have, almost surely,
$$
 X(t) = S(t)\xi +  \bS\ast F(\cdot, X) + \bS\diamond G( \cdot, X).
$$
\end{enumerate}
Here, we used the notation
\[ \bS\ast f(t):= \int_0^t S(t-s)f(s)\, ds \]
and
\[ \bS\diamond \Phi(t):= \int_0^t S(t-s)\Phi(s)\,dW_H(s)\] 
for deterministic and stochastic convolutions, respectively.

We recall that for $0\le a<b\le T$ and $\F_a$-measurable sets $A\subseteq  \O$
the stochastic integral 
of the indicator process $(t,\omega)\mapsto \one_{(a,b]\times
A}(t,\omega)\,h\otimes x$ 
with respect to $W_H$ is defined as
\[\int_0^T \one_{(a,b]\times A}\otimes [h\otimes x]\,dW_H:= \one_A \,W_H(\one_{(a, b]}\otimes h) 
\otimes x.\]
This definition extends to finite linear combinations of
adapted indicator processes of the above form. For such processes 
$\Phi$  we have the following two-sided estimate.

\begin{prop}\cite[Theorem 5.9]{vNVW07}\label{p.NVW}
Let $E$ be a UMD Banach space and let $1<p<\infty$ be fixed.
Then
$$
 \E \Big\n \int_0^T \Phi\,dW_H\Big\n^p \eqsim \E\n
\Phi\n_{\g(L^2(0,T;H),E)}^p,
$$
with implied constants depending only on $p$ and $E$.
\end{prop}
By a density argument, this `It\^o isomorphism' extends to the Banach space 
$L_\FF^p(\O;\g(L^2(0,T;H),E))$ of all $\FF$-adapted processes in
$L^p(\O;\g(L^2(0,T;H),E))$.  

Existence and uniqueness of mild solutions in suitable Banach spaces
of continuous adapted $E$-valued processes
is proved by a fixed point argument. 
Following the approach of \cite{vNVW08},
for $0\leq a < b< \infty$, $1\le p<\infty$ and $0\le \alpha < \half$
we denote by $V_\a^p([a,b]\times \Omega;E)$ the Banach 
space of all continuous adapted processes $\phi
: [a,b]\times \Omega \to E$ for which
\[\begin{aligned}
\norm{\phi}_{\alpha,p}^p := & \expect \norm{\phi}_{C([a,b];E)}^p
 + \int_a^b\expect\norm{s\mapsto (t-s)^{-\a}\phi(s)}_{\gamma (L^2(a,t),E)}^p\,dt
 \end{aligned}
\] is finite, identifying processes which are indistinguishable.

We will need the following
lemma, which allows us to estimate $\norm{}_{\a,p}$-norms  in
terms of $\norm{}_{\a,p}$-norms on smaller intervals.
 
\begin{lem}\label{l.glue}
Let $0\leq a < b < c < d$ and $\phi : [a,d]\times \Omega \to E$ be an adapted process
with $\phi \in V_\a^p([a,c]\times \Omega;E)\cap V_\a^p([b,d]\times\Omega;E)$. Then
$\phi \in V_\a^p([a,d]\times \Omega;E)$ and
\[ \n \phi \n_{V_\a^p([a,d]\times \Omega;E)} \lesssim 
\n \phi\n_{V_\a^p([a,c]\times \Omega;E)}
 + \n \phi \n_{V_\a^p([b,d]\times \Omega;E)}.
\]
\end{lem}

\begin{proof}
Clearly, $\phi$ belongs to $L^p(\Omega; C([a,d];E))$ with
\[ \n \phi\n_{L^p(\Omega; C([a,d];E))} \le 
\n \phi\n_{L^p(\Omega; C([a,c];E)}
 + \n \phi\n_{L^p(\Omega; C([b,d];E))}.
\]
Concerning the second part of the $V_{\alpha}^{p}$-norm, we 
have
\[ 
\begin{aligned}
& \int_a^b\expect\n s \mapsto (t-s)^{-\a}\phi (s)\n_{\g (L^2(a,t),E)}^p\, dt\\
&\leq \int_a^c \!\expect \n s \mapsto (t-s)^{-\a}\phi (s)\n_{\g (L^2(a,t),E)}^p dt
+ \int_c^d\! \expect\n s \mapsto (t-s)^{-\a}\phi (s)\n_{\g (L^2(a,t),E)}^p dt.
\end{aligned}
\]
Now
$$
\int_a^c \expect \n s \mapsto (t-s)^{-\a}\phi (s)\n_{\g (L^2(a,t),E)}^p\, dt
 \le \n\phi\n_{V_\a^p([a,c]\times \Omega;E)}^p 
$$
and
$$
\begin{aligned}
\ & \Big( \int_c^d \expect\n s \mapsto (t-s)^{-\a}\phi (s)\n_{\g (L^2(a,t),E)}^p\, 
dt\Big)^\frac1p
\\ & \qquad \le \Big( \int_c^d \expect \n s \mapsto (t-s)^{-\a}\phi (s)\n_{\g (L^2(a,b),E)}^p\, dt\Big)^\frac1p
\\ & \qquad\qquad +  \Big( \int_c^d \expect \n s \mapsto (t-s)^{-\a}\phi (s)\n_{\g (L^2(b,t),E)}^p\, dt\Big)^\frac1p
\\ &  \qquad \le  (c-b)^{-\a} (d-c)\expect \n \phi\n_{\g (L^2(a,b),E)} + \n\phi\n_{V_\a^p([b,d]\times \Omega;E)}.
 \end{aligned}
$$
The inequality of the first terms in the last step follows from the right ideal property
for spaces of $\g$-radonifying operators. 
Now observe that
\[
\begin{aligned}
\expect \n\phi &\n_{\g (L^2(a,b),E)}^p  = \frac{1}{c-b} \int_b^c \expect \n \phi \n_{\g(L^2(a,b),E)}^p\, dt
\leq  \frac{1}{c-b} \int_b^c \expect \n \phi \n_{\g(L^2(a,t),E)}^p\, dt\\
& \leq \frac{c^{\a p}}{c-b}\int_b^c \expect \n s\mapsto (t-s)^{-\a}\phi (s)\n_{\g (L^2(a,t),E)}^p\, dt 
 \leq \frac{c^{\a p}}{c-b}\n \phi\n_{V_\a^p([a,c]\times \Omega;E)}.
\end{aligned} 
\]
Here we have used covariance domination. Collecting the estimates, the claim follows.
\end{proof}

\begin{thm}[Existence and uniqueness, 
\hbox{\cite[Proposition 6.1]{vNVW08}}]\label{t.solution}
Let $E$ be a UMD space, and suppose that assumptions {\rm (A), (F)} and {\rm
(G)} are satisfied.
Fix $2<p<\infty$ and $\frac1p<\alpha<\half$ and let $\xi\in
L^p(\Om,\F_0;E)$ be given.
The mapping $$\cL_{\xi,T} : \phi \mapsto S(\cdot)\xi + \bS*F(\cdot, \phi)
+ \bS\diamond G(\cdot, \phi)$$
defines a Lipschitz continuous mapping on the space $\Vap$. 
Its Lipschitz constant is independent of $\xi$ and depends on $A, F, G$ only
through the constants $L_F$, $L_G$, and
the sectoriality constants of $A$, and tends towards $0$ as $T \downarrow 0$.
\end{thm}

For small $T_0>0$, the mapping $\cL_{\xi,T_0}$ has a unique fixed point in $\Vapo$,
and this fixed point turns out to be a mild solution of \eqref{SDE} on the interval $[0,T]$. 
Repeating this argument inductively in conjunction with  
Lemma \ref{l.glue}, one obtains a unique solution $X$ in $\Vap$ of 
\eqref{SDE} on the interval $[0,T]$ (see \cite[Theorem 6.2]{vNVW08}). 

We note that for $1 \leq q \leq p < \infty$ and $0\leq \a < \half$ we have a
continuous embedding
\[ \Vap \inject \Vaq .\]
Furthermore, for $1\leq p < \infty$ and $0\leq \a < \beta < \half$, 
the ideal property yields a continuous embedding
\[ V_\beta^p([0,T]\times\Omega;E) \inject \Vap .\]
These embeddings imply consistency of solutions for different values of $\a$ and $p$.\smallskip

The next lemma provides a way to test whether a given process belongs to $\Vap$.
By $L^p_{\mathbb{F}}(\Omega; C^{\lambda}([0,T];E))$ we denote the Banach space
of all continuous adapted processes $\phi : [0,T]\times \Omega \to E$ belonging to 
$L^p(\Omega; C^{\lambda}([0,T];E))$, once more identifying processes which are indistinguishable. 

\begin{lem}\label{l.embed}
Let $2<p<\infty$ and $\frac{1}{p}<\alpha<\frac12$, and let $E$ be a Banach space
with type
$\tau \in [1,2)$. Then for $\lambda > \frac{1}{\tau} - \half$ we have a 
continuous embedding
\[ L^p_{\mathbb{F}}(\Omega; C^{\lambda}([0,T];E))\inject \Vap. \] 
\end{lem}

\begin{proof}
Pick $q>0$ such that $\alpha < \half - \frac{1}{q}$. By \cite[Lemma 3.3]{vNVW08}
and the fact that 
$C^\lambda ([0,T];E)$ embeds into the Besov space $B^{\frac{1}{\tau} -
\half}_{q,\tau}(0,T;E)$, 
for all $f\in C^{\lambda}([0,T];E)$ we have
\[ \sup_{t\in (0,T)} \norm{s\mapsto (t-s)^{-\a}f(s)}_{\gamma
(L^2(0,t),E)}^p\lesssim 
\norm{f}_{B^{\frac{1}{\tau}-\half}_{q,\tau}(0,T;E)}\lesssim 
 \norm{f}_{C^{\lambda}([0,T];E)}.
\]
Thus, considering adapted $\phi \in L^p(\Omega;
C^{\lambda}([0,T];E))$ we see that
\[
 \int_0^T \expect\norm{s\mapsto (t-s)^{-\a}\phi(s)}_{\gamma (L^2(0,t),E)}^p\, dt
  \lesssim T\norm{\phi}_{L^p(\Omega;C^\lambda([0,T];E))}^p.
\]
The other part of the norm of $\Vap$ can clearly be estimated by 
the norm of $L^p(\Omega; C^{\lambda}([0,T];E)$. 
\end{proof}

\section{Continuous dependence on the coefficients}\label{sect.dependence}

We now take up our main line of study and approximate simultaneously the coefficients $A$, $F$, $G$ and
the initial datum $\xi$ in equation \eqref{SDE}.
Regarding the approximation of $A$ we make the following assumptions:

\begin{enumerate}
\item[{\rm (A1)}] The operators $A$ and $A_n$ are densely defined, closed, and
{\em uniformly  sectorial} on 
$E$ in the sense there exist numbers $M\ge 1$ and $w\in\R$ such that $A$ and
each $A_n$ is sectorial of type 
$(M,w)$.
\end{enumerate}
\begin{enumerate}
\item[{\rm (A2)}]  The operators $A_n$ converge to $A$ {\em in the strong
resolvent sense}:
$$\limn R(\la,A_n)x = R(\la,A)x$$ for all $\Re\la>w$ and $x\in E$.
\end{enumerate}

Under (A1), the operators $A$ and $A_n$ generate strongly continuous analytic
semigroups
$\bS$, $\bS_n$ satisfying the uniform bounds
$$ \begin{aligned} 
\n S(t)\n, \ \n S_n(t)\n & \le Me^{wt}, && t\ge 0,\\ 
\n AS(t)\n, \ \n A_nS_n(t)\n & \le \frac{M'}{t}e^{wt}, && t>0.
\end{aligned}
$$
The following Trotter-Kato type approximation theorem is well known; see 
\cite[Theorem 3.6.1]{abhn} for
part (1), part (2) follows from a contour integral argument.

\begin{lem}\label{l.trotterkato}
Assume {\rm (A1)} and {\rm (A2)}.
\begin{enumerate}
 \item For all $t\in [0,\infty)$ and $x \in E$ we have $S_n(t)x \to S(t)x$,  and
the convergence is uniform on compact subsets of $[0,\infty )\times E$.
 \item For all $t\in (0,\infty)$ and $x \in E$ we have
$A_n S_n(t)x \to A S(t)x$, and the convergence is uniform on compact subsets of
$(0,\infty )\times E$.
\end{enumerate}
\end{lem}

As a consequence we see that under {\rm (A1)} and {\rm (A2)}, for each $\b\in
[0,1)$ the functions 
$M_n^\b(t) := t^\beta S_n(t)$ and $M^\b(t) := t^\beta S(t)$ satisfy the
hypotheses of Lemma \ref{l.gammaconv}. Indeed,
condition \ref{l.gammaconv}(1) is clear, condition \ref{l.gammaconv}(2) follows from Lemma 
\ref{l.trotterkato}, and condition \ref{l.gammaconv}(3) 
follows from Lemma \ref{l.S} according to which 
the sets $$\{s^\b S_n(s): \ s\in (0,t)\} $$
are $\g$-bounded in $\calL(E)$ with a uniform $\g$-bound of order
$t^\b$.

We will make the following assumptions on the nonlinearities $F$ and $F_n$:
\begin{enumerate}
\item[(F1)] The maps $F$, $F_n : [0,T]\times \Omega \times E \to E$ are
uniformly Lipschitz continuous and of 
linear growth in the sense that they satisfy (F) with uniform Lipschitz and
growth constants. Furthermore, $F$ and $F_n$ satisfy the measurability assumption in (F).
\item[(F2)] For almost all $(t, \omega) \in [0,T]\times \Omega$ we have $F_n(t,
\omega, x) \to F(t,\omega, x)$ in $E$
for all $x \in E$.
\end{enumerate}
Similar assumptions are made on $G$ and $G_n$:
\begin{enumerate}
\item[(G1)] The maps $G$, $G_n : [0,T]\times \Omega \times E  \to \gamma (H, E)$
are uniformly 
$\gamma$-Lipschitz continuous and of linear growth in the sense that they
satisfy (G) 
with uniform $\g$-Lipschitz and growth constants. Furthermore, $G$ and $G_n$
satisfies the measurability assumption in (G).
\item[(G2)] For almost all $(t,\omega) \in [0,T]\times \Omega$ we have\, $G_n(\cdot,
\omega, x) \to 
G(\cdot, \omega, x)$ in $\g(L^2(0,T, \mu;H),E)$, for all $x \in E$ and all finite
measures $\mu$ on $[0,T]$.
\end{enumerate}

We will need a lemma on convergence of random variables with values in spaces of
H\"older continuous functions.
For $\eta\in L^p(\Omega; C^\lambda ([0,T];E))$
we denote by $\eta_t\in L^p(\Omega)$ the random variable $(\eta_t)(\omega):=
(\eta(\omega))(t).$

\begin{lem}\label{l.hoelderconvergence}
Let $E$ be a Banach space, let $1<p<\infty$ and $\lambda >0$, and suppose that
$\eta_n, \eta \in 
L^p(\Omega; C^\lambda ([0,T];E))$ satisfy
\begin{enumerate}
 \item $\sup_n \n\eta_n\n_{L^p(\Omega; C^\lambda ([0,T];E))}<\infty$;
 \item $(\eta_n)_t \to \eta_t$ in measure for all $t \in [0,T]$;
\end{enumerate}
Then $\eta_n \to \eta$ in $L^q(\Omega; C^\mu ([0,T];E))$ for all $1 \leq q < p$
and
$0\le \mu<\lambda$.
\end{lem}

\begin{proof}
We fix $0< \mu < \lambda$ and put $\zeta_n := \eta_n - \eta$.

Let $M:= \sup_n\norm{\zeta_n}^p_{L^p(\Omega; C^\lambda ([0,T];E))}$. 
Chebyshev's inequality implies that
\[ \sup_{n} \P\big( \norm{\zeta_n}_{C^\lambda ([0,T];E)} \geq R\big) \leq
R^{-p}M \to 0
 \quad\mbox{as}\,\, R \to \infty.
\]

It follows from this and assumption (2) that $\zeta_n \to 0$ in measure in
$C^\mu ([0,T];E)$.
The proof is the same as that of \cite[Proposition 2.1]{mss94} where the finite
dimensional situation
was considered. See also \cite[Lemma A.1]{bmss95} for further convergence
results of this form.

As a consequence of the boundedness of $(\zeta_n)$ in
$L^p(\Omega;C^\lambda([0,T];E))$ the random variables 
$\norm{\zeta_n}_{C^\lambda ([0,T];E)}^q$ -- and hence also the random variables
$\norm{\zeta_n}_{C^\mu ([0,T];E)}^q$ -- 
are uniformly integrable for all $1\leq q < p$. 
It follows from \cite[Proposition 4.12]{kallenberg} that $\zeta_n \to 0$ in
$L^q(\Omega; C^\mu ([0,T];E))$.
\end{proof}

We are now in a position to state and prove the main abstract result of this
paper. 
In its formulation we use that UMD Banach
spaces have nontrivial type. In fact, UMD Banach spaces are super-reflexive,
super-reflexive spaces are 
$K$-convex, and $K$-convexity is equivalent to having nontrivial type. For more
details and references 
to the literature we refer to \cite{handbook1, handbook2}.

In what follows we will consider $0\le \a<\frac12$ and $2< p<\infty$ to be fixed and write $$X = \sol
(A,F,G,\xi )$$ 
to indicate that $X$ is the unique mild solution
of \eqref{SDE} in $\Vap$ with coefficients $(A,F,G)$ and initial datum $\xi$.
 
\begin{thm}\label{t.coeffdep}
Let $E$ be a UMD Banach space, let $\tau\in (1,2]$ be its type,
and suppose that $\frac1p<1-\frac1\tau$. Suppose further that
the operators $A$ and $A_n$ 
satisfy {\rm (A1)} and {\rm (A2)}, the nonlinearities $F$ and $F_n$ satisfy {\rm
(F1)} 
and {\rm (F2)}, the nonlinearities $G$ and $G_n$ satisfy {\rm (G1)} and {\rm
(G2)}, and the initial data $\xi$ and $\xi_n$ satisfy  $\xi_n \to \xi$ in 
$L^p(\Omega,\F_0;E)$. Then, whenever $\frac1\tau-\frac12+\frac1p<\a<\frac12$, the 
mild solutions
$$X := \sol (A,F,G,\xi), \ \ X_n :=\sol (A_n,F_n,G_n,\xi_n)$$ satisfy
\[ X_n \to X \  \hbox{ in } \ \Vaq .\] 
for all $1\le q<p$. In particular, $X_n\to X$ in $L^q(\O;C([0,T];E))$ for all $1\le q<p$.
\end{thm}

We structure the proof through a series of lemmas.

\begin{lem}\label{l.step1} 
Let $1\leq p < \infty$ and $0\leq \alpha < \half$.
Suppose that the operators $A_n$ and $A$ satisfy {\rm (A1)} and {\rm (A2)} and
that $\xi_n \to \xi $ in $L^p(\Omega, \cF_0;E)$. Then 
\[ S_n(\cdot )\xi_n \to S(\cdot )\xi\quad \mbox{in}\,\, \Vap .\] 
\end{lem}

\begin{proof}
 By Lemma \ref{l.trotterkato}(1), for every $x \in E$ we have $S_n(t)x \to S(t)x$
in $C([0,T];E)$ as $n \to \infty$.
Hence, $S_n(\cdot )\xi\to S(\cdot )\xi$ in $C([0,T];E)$ almost surely and,
noting that the semigroups $\bS_n$ 
are uniformly bounded on $[0,T]$, say by a constant $M_T$, we infer from
dominated convergence that 
\begin{equation}
\label{eq.Sn1}
\norm{S_n(\cdot )\xi - S(\cdot )\xi}_{L^p(\Omega; C([0,T];E))} \to 0.
\end{equation}
Also, $\norm{S_n(t)\xi_n - S_n(t)\xi} \leq M_T\norm{\xi_n-\xi}$, which  implies 
\[ \expect \norm{S_n(\cdot )\xi_n - S_n(\cdot )\xi}^p_{C([0,T];E)} \leq
M_T\norm{\xi_n-\xi}^p_{L^p(\Omega;E)}
\to 0.\] 
Combining these estimates we obtain  
\begin{equation}
\label{eq.Sn2}\expect \norm{S_n(\cdot )\xi_n - S(\cdot )\xi}_{C([0,T];E)}^p \to
0.
\end{equation}

Choose $\beta>0$ such that $\alpha + \beta < \half$ and put $M_n(t)
:= t^\beta S_n(t)$ and $M(t) : = t^\beta S(t)$. 
The finite
Borel measures $\mu_t^\alpha$ on 
$(0,t)$ are defined by
\[ \mu_t^\alpha (A) = \int_A (t-s)^{-2\alpha }\, ds.\]
It is straightforward (see \cite{vNVW08}) to verify that 
\begin{align}\label{eq:mu_a}
\phi \in \gamma (L^2(0,t,\mu_t^\alpha);E) \ \Longleftrightarrow \  [s\mapsto
(t-s)^{-\alpha}
\phi (s)]\in \gamma (L^2(0,t),E)
\end{align}
with identical norms.
Almost surely we have
\[ 
 \norm{S_n(\cdot )\xi - S(\cdot )\xi}_{\gamma (L^2(0,t,\mu_t^\alpha),E)}
= \norm{s\mapsto (t-s)^{-\a} s^{-\b} (M_{n}(s)\xi -M (s)\xi)}_{\gamma
(L^2(0,t),E)} .
\] 
Let $\gamma_T := \sup_n\gamma\big(\{ t^{\beta}S_n(t) \, : \, 0\leq t \leq T \}\big)$.
By Lemma \ref{l.S} and assumption (A1), $\gamma_T < \infty$.
Using the observation \eqref{eq:mu_a} combined with Proposition \ref{p.gbdd} and
Lemma \ref{l.tensor}, 
almost surely we obtain, for all $t \in (0,T)$ and indices $n$,
\begin{align*}
 \norm{S_n(\cdot )\xi}_{\gamma (L^2(0,t,\mu_t^{\alpha}),E)} & \leq 
\gamma_T \norm{s\mapsto (t-s)^{-\alpha} s^{-\beta}\xi}_{\gamma (L^2(0,t),E)}\\
& =  \gamma_T \norm{s\mapsto (t-s)^{-\alpha}s^{-\beta}}_{L^2(0,t)}\norm{\xi}.
\end{align*}
The same estimate also holds with $\bS_n$ replaced with $\bS$.
Note that
\begin{align*}  
\norm{s\mapsto (t-s)^{-\alpha}s^{-\beta}}_{L^2(0,t)}^2 
 =  t^{1-2\alpha - 2\beta}\int_0^1r^{-2\beta}(1-r)^{-2\alpha}\,dr, 
\end{align*}
which is finite. Since $\alpha + \beta < \frac{1}{2}$, the supremum over $t \in
[0,T]$ of this expression  is bounded,
say by $C_{T,\alpha, \beta}$. Hence we have
\[ \E\norm{S_n(\cdot )\xi - S(\cdot )\xi}_{\gamma (L^2(0,t,\mu_t^\alpha ),E)}^p
\leq  2 \g_T^p C_{T,\alpha, \beta}^p \norm{\xi}_{L^p(\Omega; E)}^p.
\]

Furthermore, by the observation following Lemma \ref{l.trotterkato}  we may
apply Lemma \ref{l.gammaconv} to the 
functions $M_n, M$ and the $\g$-radonifying operators represented by the
functions 
$s\mapsto (t-s)^{-\a}s^{-\b}\xi(\om)$ to conclude that
\[ \norm{S_n(\cdot )\xi - S(\cdot )\xi}_{\gamma (L^2(0,t,\mu_t^\alpha ),E)} \to
0 \]
almost surely.

Hence, by dominated convergence, 
\[ \int_0^T\expect \norm{S_n(\cdot )\xi - S(\cdot )\xi}_{\gamma
(L^2(0,t,\mu_t^\alpha ),E)}^p\, dt
 \to 0 .
\]
Together with \eqref{eq.Sn1} this shows that
 $S_n(\cdot)\xi \to S(\cdot )\xi$ in $\Vap$.

Arguing as before we see that
\[ \norm{S_n(\cdot )\xi_n - S_n(\cdot )\xi}_{\gamma (L^2(0,t,\mu_t^\alpha ),E)}
 \leq \gamma_T\norm{s\mapsto
(t-s)^{-\alpha}s^{-\beta}}_{L^2(0,t)}\norm{\xi_n-\xi},
\]
so
\[ \int_0^T\expect \norm{S_n(\cdot )\xi_n - S_n(\cdot )\xi}_{\gamma
(L^2(0,t,\mu_t^\alpha ),E)}^p\, dt
 \lesssim \norm{\xi_n -\xi}_{L^p(\Omega; E)}^p \to 0 .
\]
Combining this with \eqref{eq.Sn2} this gives
$ \n S_n(\cdot)\xi_n - S_n(\cdot)\xi\n \to 0$ in $\Vap$.

Collecting the estimates, the proof is complete.
\end{proof}

\begin{lem}\label{l.step2}
Let $E$ be a UMD Banach space and assume {\rm (A1), (A2), (F1), (F2), (G1)} and {\rm (G2)}.
Suppose that $\frac1p < \a < \half$ and let $\phi \in \Vap$ be given. Then, for all $0\le \lambda < \alpha - \frac1p$
and $1\leq q < p$ we have
\begin{enumerate}
 \item $\bS_n\ast F_n(\cdot, \phi ) - \bS_n \ast F(\cdot,\phi ) \to 0$ in $L^p(\Omega; C^\lambda([0,T];E ))$;
 \item $\bS_n\diamond G_n(\cdot, \phi ) - \bS_n \diamond G(\cdot,\phi ) \to 0$ 
in $L^p(\Omega; C^\lambda([0,T];E ))$;
\item $\bS_n\ast F(\cdot, \phi ) \to \bS \ast F(\cdot,\phi )$ in $L^q(\Omega; C^\lambda([0,T];E ))$;
\item $\bS_n\diamond G(\cdot, \phi ) \to \bS \diamond G(\cdot,\phi )$ in $L^q(\Omega; C^\lambda([0,T];E ))$.
\end{enumerate}

\end{lem}

\begin{proof}
 (1) Let us denote the fractional convolution operator of exponent $0<a<1$ associated
with $A_n$ by $I_{a,n}$: 
$$
 I_{a,n}f(t) := \frac{1}{\Gamma (a)}\int_0^t (t-s)^{a-1}S_n(t-s)f(s)\, ds. 
$$
Pick any $r>2$ and, noting that $\frac{1}{r}+ \lambda  < 1$, $a$ such that
$\frac{1}{r} + \lambda < a< 1$. Then
$$\bS_n \ast f = I_{a,n}I_{1-a,n}f$$ for all $f \in L^r(0,T;E)$. Moreover,
by the uniform sectoriality of the operators $A_n$, the operators $I_{a,n}$ 
are uniformly bounded on $L^r(0,T;E)$
 and uniformly bounded from $L^r(0,T;E)$ to
$C^\lambda([0,T];E)$ (cf. \cite{dpkz,dprzab2}). Hence, 
\[
\begin{aligned}
 \mathbb{E}\norm{\bS_n\ast F_n(\cdot,\phi) - \bS_n\ast F(\cdot,\phi
)}_{C^\lambda([0,T];E)}^p
 &\lesssim \mathbb{E}\norm{I_{1-a,n} (F_n(\cdot,\phi ) - F(\cdot,\phi
))}_{L^r(0,T;E)}^p\\ 
&\lesssim \mathbb{E}\norm{F_n(\cdot,\phi ) - F(\cdot,\phi ))}_{L^r(0,T;E)}^p,
\end{aligned}
\]
Now note that
\[ \norm{F_n(t, \omega, \phi(t,\omega )) - F(t, \omega, \phi(t, \omega ))} \leq
2C_F(1+ \norm{\phi(t, \omega)}),
\]
and the right-hand side belongs to $L^p(\Omega; L^r(0,T;E))$. 
Since for almost all $(t,\omega)$ we have $F_n(t, \omega, \phi(t,\omega )) \to
F(t, \omega, \phi (t,\omega ))$, 
 $\mathbb{E}\norm{F_n(\cdot, \phi ) - F(\cdot, \phi )}_{L^r(0,T;E)}^p\to 0$ follows by
dominated convergence. 
Thus $\mathbb{E}\norm{\bS_n\ast F_n(\cdot,\phi) - \bS_n\ast F(\cdot, \phi
)}_{C^\lambda([0,T];E)}^p\to 0$ as $n\to\infty$.\medskip

(2) By  \cite[Proposition 4.2]{vNVW08},
\[
\begin{aligned} 
&\norm{\bS_n\diamond G_n(\cdot, \phi ) - \bS_n\diamond G(\cdot,\phi
)}_{L^p(\Omega; C^\lambda([0,T];E ))}\\
& \qquad  \lesssim \Big(\int_0^T\expect\norm{G_n(\cdot, \phi) - G(\cdot, \phi
)}_{\gamma (L^2(0,t,\mu_t^\alpha;H ),E)}^p
\, dt \Big)^{\frac{1}{p}}.
\end{aligned}
\]
We note that for almost all $\omega$ we have
\[
 \norm{G_n (\cdot, \omega, \phi) - G(\cdot, \omega, \phi)}_{\g
(L^2(0,t,\mu_t^\alpha;H),E)}
\leq C_G(1+ \norm{\phi (\cdot, \omega )}_{L^2_\g(0,t,\mu_t^\alpha;E)}) .
\]
The right-hand side belongs to $L^p((0,T)\times \Omega )$ since $\phi
\in \Vap$. Hence if we prove
that $G_n(\cdot, \omega, \phi (\cdot, \omega )) \to G(\cdot, \omega, \phi
(\cdot, \omega ))$
in $\g (L^2(0,t,\mu_t^\alpha;H),E)$ for almost all $t$ and $\omega$, then, by
dominated
convergence, we conclude that $\bS_n\diamond G_n(\cdot, \phi ) - \bS
\diamond G(\cdot, \phi ) \to 0$
in $L^p(\Omega; C^\lambda ([0,T];E))$. 

Fix $t\in [0,T]$ and $\omega \in \Omega$. For notational
convenience 
we shall suppress the dependence on $\omega$. Let $\psi:= \sum_{k=1}^K \one_{A_k}x_k$
be a simple function.
Then $$G(\cdot, \psi (\cdot )) = \sum_{k=1}^K\one_{A_k}G(\cdot, x_k)$$ and
similarly for $G_n$. 
Note that $G_n(\cdot, x_k) \to G(\cdot, x_k)$ in $\g
(L^2(0,t,\mu_t^\alpha;H),E)$ for all 
$1\leq k \leq K$ by assumption (G2). Since multiplication by $\one_{A_k}$ is a
bounded operator 
on $L^2(0,t,\mu_t^\alpha;H)$
it follows from the ideal property of $\g$-radonifying operators that
$\one_{A_k}(\cdot)G(\cdot, x_k)
\to \one_{A_k}(\cdot)G(\cdot, x_k)$ in $\g (L^2(0,t,\mu_t^\alpha;H),E)$ for all
$1\leq k \leq K$. 
Summing up it follows that $G_n(\cdot, \psi ) \to G(\cdot, \psi )$ in $\g
(L^2(0,t,\mu_t^\alpha;H),E)$.

Now let $\psi, \psi_0 \in \g (L^2(0,t;\mu_t^\alpha),E)$ be arbitrary. By
assumption (G1) we have
\[ 
\begin{aligned}
&\norm{G_n(\cdot, \psi) - G(\cdot, \psi )}_{\g (L^2(0,t;\mu_t^\alpha;H),E)}\\
&\qquad  \leq 2L_G\norm{\psi - \psi_0}_{L^2_\g(0,t;\mu_t^\alpha;E)} +
\norm{G_n(\cdot, \psi_0 ) - 
G(\cdot, \psi_0 )}_{\g (L^2(0,t,\mu_t^\alpha;H),E)}.
\end{aligned}
\]
Thus, by first choosing a simple function $\psi_0$ close enough to $\psi$ and
then $n$ large enough, we see
that for any $\psi \in \g (L^2(0,t;\mu_t^\alpha),E)$ we have $G_n(\cdot, \psi )
\to G(\cdot, \psi )$ 
in $\g (L^2(0,t,\mu_t^\alpha;H),E)$. 

Using this result pathwise, it follows that, almost surely, 
$G_n(\cdot, \phi ) \to G(\cdot, \phi )$ in $\g (L^2(0,t,\mu_t^\alpha;H),E)$.
Since $t$ was arbitrary,
this finishes the proof of (2).\medskip

(3) We pick $\mu$ such that $\lambda < \mu < \a -\frac1p$. 
Arguing as in the proof of (1), for large $r$ we have
\[\begin{aligned}
   \expect\|\bS_n\ast F(\cdot,\phi ) - \bS\ast F(\cdot,\phi )\|_{C^\mu
([0,T];E)}^p
&\lesssim 2\expect\norm{F(\cdot, \phi )}_{L^r(0,T;E)}^p\\
& \lesssim \expect \Big(\int_0^T (1+ \norm{\phi (t)})^r\,
dt\Big)^{\frac{p}{r}} < \infty.
  \end{aligned}
\]
Now observe that, almost surely, we have
\[ \int_0^tS_n(t-s)F(s,\phi (s))\, dt \to \int_0^tS(t-s)F(s,\phi (s))\, ds\]
in $E$ for every $t \in [0,T]$, by  dominated convergence. 
Applying the dominated convergence theorem a second time, we see that
$\big[\bS_n\ast F(\cdot, \phi )\big](t) \to \big[\bS\ast F(\cdot, \phi
)\big](t)$ in $L^p(\Omega; E)$. 
Convergence of the deterministic convolution in $L^q(\Omega; C^\lambda ([0,T];E))$
follows from Lemma \ref{l.hoelderconvergence}.\medskip

(4) Arguing similarly as in the proof of (3) we see that
$\bS_n\diamond G(\cdot, \phi ) - \bS\diamond G(\cdot, \phi )$ is bounded in
$L^p(\Omega; C^\mu ([0,T];E))$ for $\lambda < \mu < \a - \frac1p$.

Now fix $t \in [0,T]$. Since $s\mapsto (t-s)^{-\alpha}G(s, \phi(s))$ belongs to
$\gamma (L^2(0,t;H),E)$) almost surely, it follows from Lemma \ref{l.gammaconv} as in the proof of Lemma 
\ref{l.step1} that,
almost surely, 
$S_n(t-\cdot )G(\cdot, \phi (\cdot )) \to S(t-\cdot )G(\cdot, \phi (\cdot ))$ in
$\gamma (L^2(0,t;H),E)$.
Furthermore, by Proposition \ref{p.gbdd}, 
\begin{align*}
 &  \norm{S_n(t-\cdot )  G(\cdot,\phi (\cdot ))}_{\gamma (L^2(0,t;H),E)}\\
 & \qquad \leq \gamma_{\alpha,T}\norm{(t-\cdot )^{-\alpha}G(\cdot, \phi (\cdot
))}_{\gamma (L^2(0,t;H),E)}
\leq \gamma_{\alpha,T} C_G(1+\norm{\phi}_{\alpha,p}) ,
\end{align*}
where $\gamma_{\alpha ,T}:= \sup_n\g (\{s^\alpha S_n(s)\, : \, 0<s \leq T\}) <
\infty$ by Lemma
\ref{l.S} and the uniform sectoriality of the operators $A_n$.
Hence, by dominated convergence and the It\^o isomorphism,
\[
\begin{aligned}
& \E \n [S_n\diamond G(\cdot,\phi)](t)- [S\diamond G(\cdot, \phi)](t)\n^p\\
&\qquad \eqsim \E \n S_n(t-\cdot )  G(\cdot,\phi (\cdot )) - S(t-\cdot)G(\cdot, \phi
(\cdot ))\n_{\gamma (L^2(0,t;H),E)}\to 0.
\end{aligned}
\]
Now Lemma \ref{l.hoelderconvergence} yields $\bS_n\diamond G(\phi ) \to
\bS\diamond G(\phi )$
in $L^q(\Omega; C^\lambda ([0,T];E))$.
\end{proof}

\begin{proof}[Proof of Theorem \ref{t.coeffdep}]
We may replace $q$ be some larger value and thereby assume that $q\in (2,p)$ and
$\a\in(\frac{1}{q},\half)$.\medskip

{\it Step 1} -- We prove the theorem for small $T_0$.

Let $\Lambda$ and $\Lambda_n$ denote the Lipschitz continuous mappings 
on $\Vaq$ used to solve \eqref{SDE} with data $(A,F,G,\xi)$ and $(A_n,F_n,G_n,\xi_n)$
respectively (see Theorem \ref{t.solution}).
We choose $T_0>0$ so small that, for some constant
$0 \leq c < 1$,
$$\sup_n \norm{\cL_n(\phi) - \cL_n(\psi )}_{\alpha,q} \leq c
\norm{\phi-\psi}_{\alpha,q}$$ 
for all $\phi, \psi \in V_{\alpha}^q([0,T_0]\times\Omega;E)$. This is possible
by Theorem \ref{t.solution}, noting that all estimates involving $A_n, F_n, G_n$ are uniform in $n$.  

We denote by $X$ and $X_n$ the unique fixed points in $V_\a^p([0,T_0]\times \Omega;E)$ of the operators
$\cL$ and $\cL_n$, so that
\begin{equation}\label{eq.lambda}
\begin{aligned}
X= \cL(X) & = S(\cdot )\xi + \bS\ast F(\cdot,X) + \bS\diamond G(\cdot,X),\\ 
X_n= \cL(X_n) &= S_n(\cdot )\xi_n + \bS_n\ast F_n(\cdot,X_n) + \bS_n\diamond G_n(\cdot,X_n).
\end{aligned}\end{equation}
We have
\begin{align*}
 \norm{X-X_n}_{\alpha,q} =  \norm{\cL(X) - \cL_n(X_n)}_{\alpha,q}
   \leq  \norm{\cL(X) - \cL_n(X)}_{\alpha,q} 
+ c\norm{X-X_n}_{\alpha,q},
\end{align*}
and therefore
$$ \norm{X-X_n}_{\alpha,q}\le 
(1-c)^{-1}\norm{\Lambda(X)-\Lambda_n(X)}_{\alpha,q}.
$$
Hence, in order to prove that $X_n\to X$ in $V_\a^q([0,T_0]\times\Omega;E)$ it suffices to prove that
$\cL_n (X) \to \cL(X)$ in $V_\a^q([0,T_0] \times\Omega;E)$. But this follows from Lemma \ref{l.step1} and, picking
$\lambda$ such that $\frac{1}{\tau}-\half < \lambda < \a -\frac1p$, from Lemma \ref{l.step2}
and the embedding of Lemma \ref{l.embed}.\medskip

{\it Step 2} -- We prove the result for general $T$.

Let $T_0$ as in Step 1. By Lemma \ref{l.glue}, we have
\[ \n X-X_n\n_{V_\a^q([0,\frac{3}{2}T_0]\times\Omega;E)} \lesssim
 \n X-X_n\n_{V_\a^q([0,T_0]\times\Omega;E)} + \n X-X_n\n_{V_\a^q([\half T_0,\frac{3}{2}T_0]\times\Omega;E)}.
\]
By Step 1, the first term on the right-hand side converges to 0 as $n \to \infty$. 
But so does the second term,
noting that $X$ and $X_n$ are the unique solutions of the `shifted' equations 
starting at initial time $\half T_0$
with initial values $X(\half T_0)$ and $X_n(\half T_0)$ respectively. 

Inductively, we obtain convergence in $V_\a^q([0,(1+\frac{k}{2})T_0]\times\Omega;E)$
for all $k \in \CN$ and hence in $\Vaq$ for all times $T$.
\end{proof}

\begin{rem}
Assume the hypotheses of Theorem \ref{t.coeffdep} and additionally that
$A_n \equiv A$. Then $$X_n \to X \ \hbox{ in }\Vap$$ 
(rather than only in $\Vaq$ for $1\le q<p$). In particular, the solution of the equation with
fixed coefficients $A, F$ and $G$ depends continuously on the
initial datum $\xi\in L^p(\O,\F_0;E)$ in the norm of  $\Vap$. 
This follows by repeating the proof of Theorem \ref{t.coeffdep}, and observing that
this time parts (3) and (4) of Lemma \ref{l.step2} are not needed.
\end{rem}

The second part of 
Theorem \ref{t.coeffdep} asserts that $X_n \to X$ in $L^q(\Omega; C([0,T];E))$. We will show next
that the `compensated solutions' even converge in the norm of $L^q(\Omega; C^\lambda([0,T];E))$.

\begin{thm}\label{t.hoelderconv}
Under the assumptions of Theorem \ref{t.coeffdep}, for all $0\le \lambda<  \half - \frac{1}{p}$
we have
\[ X_n - S_n(\cdot )\xi_n \to X - S(\cdot )\xi \ \hbox{ in } \ 
L^q(\Omega; C^\lambda ([0,T]; E))\]
 for all $1 \leq q < p$. 
\end{thm}

\begin{proof}
We may 
assume that $\frac1\tau-\frac12<\lambda<\frac12-\frac1p$. 
Choose $0<\a<\frac12$ in such a way that $\frac1\tau-\frac12<\lambda<\a-\frac1p$.

Let $\Lambda_0$ and $\Lambda_{n,0}$ denote the Lipschitz continuous mappings used to
solve \eqref{SDE} with data $(A,F,G, 0)$ and $(A_n,F_n,G_n,0)$ respectively, i.e.
they are given as in \eqref{eq.lambda} with $\xi_n\equiv \xi = 0$.
We have
\[ 
\begin{aligned}
\ & \n X - S(\cdot )\xi - X_n + S_n(\cdot )\xi_n \n_{L^q(\Omega; C^\lambda ([0,T];E))}\\
& \ \  \leq \n \Lambda_0(X) - \Lambda_{n,0}(X)\n_{L^q(\Omega; C^\lambda ([0,T];E))} 
+\n \Lambda_{n,0}(X) - \Lambda_{n,0}(X_n)\n_{L^q(\Omega; C^\lambda ([0,T];E))}.
\end{aligned}
\]
As a direct consequence of Lemma \ref{l.step2},  $\Lambda_{n,0}(X) \to \Lambda_0(X)$ in 
$L^q(\Omega; C^\lambda ([0,T];E))$.

Combining a standard factorization argument (e.g., 
as in the proof of \cite[Theorem 6.2]{vNVW08}) with the assumptions on $F_n$ and $G_n$, 
one sees that the mappings $\Lambda_{n,0}$ are Lipschitz continuous  from
$\Vap$ to $L^q(\Omega; C^\lambda ([0,T];E))$, with uniformly bounded Lipschitz constants. 
Thus
\[ \n \Lambda_{n,0}(X) - \Lambda_{n,0}(X_n)\n_{L^q(\Omega; C^\lambda ([0,T];E))}
 \lesssim \n X-X_n\n_{\a,q} \to 0
\]
by Theorem \ref{t.coeffdep}. This finishes the proof.
\end{proof}

\section{Applications}\label{sect.applications}

\subsection{Approximating the noise}\label{subsec.noise} 
As a first application we show that if $H$ is separable, we may always
approximate the 
cylindrical Brownian motion $W_H$ with finite dimensional noise. The strategy is
to choose a sequence
of projections $P_n$ on $H$ with finite dimensional ranges which converges
strongly to the identity.
Then we approximate the map $G$ by the functions $G_n := GP_n$.

For $M$-type 2 spaces, such approximations were considered in \cite{Brz97}.
In order to apply our results from the previous section
we must check that (G1) and (G2) hold.

Assumption (G1) follows from
the ideal property of $\gamma$-radonifying operators and the uniform boundedness
of the 
projections $P_n$. Assumption (G2) is an immediate consequence of
\cite[Proposition 2.4]{vNVW07}. 

\subsection{Yosida approximations}\label{subsec.Yosida}
As we have already mentioned in the introduction, from a theoretical point of
view it is 
useful to be able to approximate the generator $A$
by its {\em Yosida approximands} $A_n := n^2R(n,A) -n$.

For Hilbert spaces $E$, 
Yosida approximations for stochastic evolution equations are considered 
in Da Prato and Zabczyk \cite{dprzab2}
(see also \cite{Bier} for an expanded argument), where continuous dependence in $L^p(\O;C([0,T];E))$ is obtained 
without analyticity assumptions on $A$. 

In order to apply our results we must check that assumptions (A1) and (A2) hold
for these operators. The uniform sectoriality (A1) follows from \cite[Proposition 2.1.1 (f)]{haase}.
As for the strong resolvent convergence (A2), we note that the standard proof of
the Hille-Yosida
theorem is to prove that the semigroups $\bS_n$ generated by $A_n$ are uniformly
exponentially
bounded and converge strongly to the semigroup $\bS$ generated by $A$. Taking
Laplace transforms,
the strong resolvent convergence follows. See also \cite[Section 3.6]{abhn}.

\subsection{Approximating the coefficients in parabolic
SPDEs}\label{subsect:global}

In this section we apply our results to stochastic partial differential
equation.
For simplicity, we confine ourselves to the situation where the nonlinearities
$f$ and $g$
are time-independent and consider equations of the form 
$$
\left\{
\begin{aligned}\displaystyle
\frac{\partial u}{\partial t}(t,x ) &= \mathcal{A}u (t,x) + f(u(t,x)) +
\sum_{k=1}^K g_k(u(t,x))\frac{\partial W_k}{\partial t} (t), 
&& x \in \mathcal{O}, && t>0,
\\
 u(t,x) &= 0, 
&& x \in \partial \mathcal{O}, &&  t>0, 
\\
 u(0,x) &= \xi(x), && x \in \mathcal{O} .&&
\end{aligned}\right.
$$
Here  $\mathcal{O}$ is a bounded open domain in $\R^d$ and $\mathcal{A}$ is the
second order divergence form differential operator 
$$\mathcal{A}u(x)= \sum_{i=1}^d \frac{\partial}{\partial
x_i}\Big(a_{ij}(x)\sum_{j=1}^d\frac{\partial u}{\partial x_j}(x)\Big)
+ \sum_{j=1}^d b_j(x) \frac{\partial u}{\partial x_j}(x).$$
The driving processes $W_k$ are independent real-valued standard Brownian
motions.
Under the assumptions of Theorem \ref{thm:main} we would like to approximate the coefficients 
${\bf a}=(a_{ij})$ and ${\bf b} =
(b_j)$ as well as the functions $f$ and $g_k$, and study the convergence of the
approximate solutions to the exact solution.

In order to reformulate the above SPDE as a stochastic Cauchy problem on the
Banach space
$L^r(\mathcal{O})$ (we use the exponent $r$ since the exponents $p$ and $q$ have
already been used in a different meaning)
we use a variational approach.
Consider the sesquilinear form
\begin{equation}\label{eq.form}
 \mathfrak{a}[u,v] := \int_{\mathcal{O}} (\mathbf{a}\nabla u) \cdot
\overline{\nabla v} +
 (\mathbf{b}\cdot \nabla u) \overline{v} \, dx
\end{equation} 
on the domain
$$ \Dom(\mathfrak{a}) := H^1_0(\mathcal{O}).$$
The sectorial operator $A$ on $L^2(\OO)$ associated with $\mathfrak{a}$
generates a strongly continuous analytic 
semigroup $\bS$, which by \cite{daners00} extrapolates to a consistent family of
strongly continuous analytic 
semigroups $\bS^{(r)}$ on $L^r(\mathcal{O})$ for $1<r<\infty$. We denote their
generators by $A^{(r)}$. 
Thus, $\bS^{(2)} = \bS$ and $A^{(2)} = A$. 
The forms $\mathfrak{a}_n$ and the associated semigroups $\bS_n^{(r)}$
with generators $A_n^{(r)}$ are defined likewise.

\begin{lem}\label{l.apA}
If {\rm (i), (ii)} and {\rm (iv)} of Theorem \ref{thm:main} hold, then the operators $A^{(r)}$ and $A_n^{(r)}$ satisfy 
{\rm
(A1)} and {\rm (A2)}.
\end{lem}

\begin{proof}
(A1): It follows from the uniform ellipticity and boundedness condition that the numerical ranges 
of the forms $\mathfrak{a}_n$ are contained in a common right open sector around
the real axis. This in turn implies that there exists a constant $c\geq 0$ and
an angle $\vartheta>0$
such that the shifted operators $A_n^{(2)} - c$ generate analytic semigroups
$T_n^{(2)}(t) = e^{-ct}S_n^{(2)}(t)$ which are uniformly bounded on the sector 
$\Sigma_\vartheta:= \{ z \in \CC\setminus \{0\}\, : \, |\arg z | < \vartheta\}$. 

Now pick $1<s<\infty$, $s\not=2$, such that $r$ lies between $s$ and $2$ and
put 
\[ w:= 2\max\{s-1, s'-1\}\kappa^{-1}C^2.\] Here, $s'$ denotes
the conjugate index to $s$, and $\kappa$ and $C$ are as in Theorem \ref{thm:main}. 
It follows from \cite[Theorem 5.1]{daners00} that
$\norm{S_n^{(s)}(t)} \leq e^{wt}$ for all 
$t\geq 0$. By taking a larger value for $c$ if necessary, it follows that
the semigroup $T_n^{(s)}(t) = e^{-ct}S_n^{(s)}(t)$ satisfies 
$\norm{T_n^{(s)}(t)} \leq 1$.

We are now in a position to use the Stein interpolation theorem 
in the version of \cite[Lemma
5.8]{kunweis}. It follows that the semigroups
$\mathbf{T}_n^{(r)}$ are uniformly bounded on a slightly smaller sector
$\Sigma_{ \vartheta'}$. 
Rescaling again, this implies that the operators $A_n^{(r)}$ are uniformly
sectorial.

(A2): It follows from (a special case of)
\cite[Theorem 5.2.4]{daners08}
that $A_n^{(2)} \to A^{(2)}$ in the strong resolvent sense. Since $\mathcal{O}$
is bounded, the embedding 
$H^1_0(\mathcal{O})\inject L^2(\mathcal{O})$ is compact. In particular,
$A^{(2)}$ has compact resolvent. It now follows from
\cite[Theorem 4.3.4]{daners08} that $A_n^{(r)} \to A^{(r)}$ in the strong
resolvent sense.
\end{proof}

To simplify notation, in what follows we fix $1<r<\infty$ and write $A :=
A^{(r)}$, $A_n := A_n^{(r)}$, and 
$\bS := \bS^{(r)}$, $\bS_n := \bS_n^{(r)}$. 

\begin{lem}\label{l.apFG}
Assume that {\rm (iii)} and {\rm (v)} of Theorem \ref{thm:main} hold.
\begin{enumerate}
\item
The maps $F, F_n:  L^r(\mathcal{O}) \to
L^r(\mathcal{O})$
defined by 
$$[F(u)](x) := f( u(x)), \quad [F_n( u)](x) := f_n(u(x)),$$ 
 satisfy {\rm (F1)} and {\rm (F2)}.
\item The maps $G, G_n : L^r(\mathcal{O}) \to
\mathscr{L}(\CR^K, L^r(\mathcal{O}))$ defined by
$$ [G(u)h](x) := \sum_{k=1}^K g_k(u(x))[e_k,h] , \quad  [G_n(u)h](x) := 
\sum_{k=1}^K g_{n,k}(u(x))[e_k,h],$$ 
where $(e_k)_{k=1}^K$ is the standard unit basis of $\CR^K$,
satisfy {\rm (G1)} and {\rm (G2)}.
\end{enumerate}
\end{lem}
\begin{proof}
(1) The assumptions imply that the maps $F$ and $F_n : L^r(\OO) \to L^r(\OO)$ are
Lipschitz continuous 
and of linear growth on $L^r(\OO)$ with uniform constants. Hence (F1) is
satisfied.

Assumption (F2) follows from dominated convergence.\medskip

(2) Let $T>0$ and $\mu$ be a finite Borel measure on $(0,T)$. If $(h_m)_{m\ge 1}$
is an orthonormal
basis of $L^2((0,T), \mu )$, then $(h_m\otimes e_k)_{m\ge 1,\, k =1,\ldots ,
K}$ is an orthonormal
basis of $L^2((0,T), \mu; \CR^K)$.  Using this fact, it is easy to see that
\[ \n R \n_{\g (L^2((0,T),\mu;\CR^K), L^r(\OO))} \leq
 \sum_{k=1}^K \n Re_k\n_{\g (L^2((0,T),\mu), L^r(\OO))}
\]
for all $R \in \g (L^2((0,T),\mu; \CR^K),L^r(\OO))$. 

This shows that it suffices to consider the case $K=1$. But in this case, we
infer from \cite[Example 5.5]{vNVW08} that 
$G$ and $G_n$ are $\gamma$-Lipschitz continuous with a uniform $\g$-Lipschitz
constant. Furthermore, by dominated
convergence, $G_n(u) \to G(u)$ in $L^r(\OO)$ for all $u\in L^r(\OO)$. 
\end{proof}

We are now ready to rewrite our SPDE as an abstract Cauchy problem 
$$ 
\begin{aligned}
dX(t)& = [AX(t) + F(X(t))]\,dt + G(X(t))\,dW(t), \quad t\in [0,T],\\ 
 X(0)& = \xi.
\end{aligned}
$$ 
We denote by $X_n$ the solution of this problem with $A$, $F$, $G$, $\xi$
replaced by $A_n$, $F_n$, $G_n$, $\xi_n$. 
Noting that the type of $L^r(\OO)$ is $\min\{r,2\}$, we deduce 
the following strengthened version of Theorem \ref{thm:main} from Theorem
\ref{t.hoelderconv}:

\begin{thm}\label{thm:5.3}
Assume that {\rm (i)} -- {\rm (v)} of Theorem \ref{thm:main} hold, and let $1<r<\infty$ and $p>\max\{2, r'\}$, where $r'$
denotes the conjugate index to $r$.
If $\xi_n\to \xi$ in $L^p(\Omega, \cF_0; L^r(\mathcal{O}))$, then
$$X_n - S_n(\cdot)\xi_n \to X - S(\cdot)\xi \ \hbox{in} \  L^q(\Omega;
C^\lambda([0,T]; L^r(\mathcal{O}))$$
for all $1\leq q < p$ and $0\leq \lambda < \half -\frac{1}{p}$.
Moreover, 
$  X_n \to X$ in ${L^q(\Omega; C([0,T]; L^r(\mathcal{O})))}$
for all $1\le q<p$.
\end{thm}

We emphasize that the applicability of our approach is by no means limited to
second order 
differential operators in divergence form and Dirichlet boundary conditions. 
Indeed, our approach applies to any sequence of operators for which the
conditions (A1) and (A2) can be verified. This includes, for instance,
the case of Neumann boundary conditions in the above example (provided 
$\partial \mathcal{O}$ is Lipschitz, so as to ensure the compactness of the
embedding $H^1(\mathcal{O)}\embed L^2(\mathcal{O})$).

\subsection{Approximating the domain}\label{subsect:domain}

A map $\bS : [0,\infty ) \to
\mathscr{L}(E)$ is called a
\emph{degenerate semigroup} if it is strongly continuous and we have $S(t+s) =
S(t)S(s)$ for all $t,s \geq 0$.
Thus the only difference to a strongly continuous semigroup is that we do not
assume that $S(0) = I$. Instead,
$\pi:= S(0)$ is now a bounded projection which commutes with every operator
$S(t)$. Consequently, we can
write $S(t) = \iota \tilde S(t)\pi$, where $\pi$ is the projection viewed as an
operator onto its range $\tilde E$,
$\tilde{\mathbf{S}}$ is the restriction 
of $\bS$ to $\tilde E$ which is invariant under $\bS$,
and $\iota : \tilde E \subseteq E$
is the canonical inclusion map.

It is easy to see that the Laplace transform of $\bS$ is given by $\iota
R(\lambda, \tilde A)\pi$ for $\Re \lambda$
large enough, where $R(\lambda, \tilde A)$ denotes the resolvent of the generator
$\tilde A$ of the restricted semigroup
$\tilde{\mathbf{S}}$. We may thus say that the \emph{generator} $A$ of $\bS$ is the
operator $\tilde A$, viewed as an operator
on $E$. 
As a replacement for the resolvent
of $A$ we define $R_\lambda (A) := \iota R(\lambda, \tilde A) \pi$ and note that
$R_\lambda (A)$ is a pseudo-resolvent.
We will call a degenerate semigroup $\bS$ \emph{analytic} if the restricted
semigroup $\tilde{\mathbf{S}}$ is analytic on $\tilde{E}$.
For more information on degenerate semigroups and pseudo-resolvents we refer
the reader to \cite{are01}.

When we allow degenerate semigroups we will assume that $A_n$ and $A$ are
generators of analytic, degenerate
semigroups $\bS_n := \iota_n\tilde{\mathbf{S}}_n\pi_n$ and $\bS := \iota
\tilde{\mathbf{S}}\pi$. 
Furthermore, we will make the following assumptions:
\begin{enumerate}
\item[$\mathrm{(A1')}$] The operators $A$ and $A_n$ are {\em uniformly 
sectorial} 
 in the sense there exist numbers $M\ge 1$ and $w\in\R$ such that 
\begin{itemize}
 \item[(i)]  $\{\lambda \in \CC\, : \,
\Re\lambda > w\}$ is contained in $\rho (\tilde{A})\cap (\bigcap_{n}\rho (\tilde{A}_n))$,
where $\tilde{A}$ and $\tilde{A}_n$
denote the generators of the strongly continuous semigroups $\tilde{\mathbf{S}}$ and
$\tilde{\mathbf{S}}_n$ respectively.
\item[(ii)] We have
\[ \sup_{\Re\la > w}\n (\la -w)R_\la (A_n)\n\le M, \,\, \sup_{\Re\la > w} \n (\la
-w)R_\la (A)\n \leq M.\]
\end{itemize}
\end{enumerate}
\begin{enumerate}
\item[$\mathrm{(A2')}$]  The operators $A_n$ converge to $A$ {\em in the strong
resolvent sense}, i.e.
$$\limn R_\la (A_n)x = R_\la (A)x$$ for all $\Re\la>w$ and $x\in E$.
\end{enumerate}

The reader may check that, \emph{mutatis mutandis}, all results of this article
extend to the degenerate case
when replacing (A1) and (A2) with $\mathrm{(A1')}$ and $\mathrm{(A2')}$. In
particular, 
there is a Trotter-Kato type theorem for degenerate semigroups, see
\cite[Theorem 5.1]{are01}.

As an application, we shall use degenerate semigroups to
study the dependence of the solutions of the SPDE in the previous section on the
domain $\OO$. Recall that a sequence of domains $\mathcal{O}_n$
is said to converge towards $ \mathcal{O}$ \emph{in the sense of Mosco} if
\begin{itemize}
\item If $(u_n)$ is a sequence in $H^1_0(\R^d)$ with $u_n \in
H^1_0(\mathcal{O}_n)$ for all $n$, then
every weak limit point of this sequence in $H^1_0(\R^d)$ lies in
$H^1_0(\mathcal{O})$;
\item For all $u \in H^1_0(\mathcal{O})$ there exists a sequence 
$(u_n)$ is a sequence in $H^1_0(\R^d)$ with $u_n \in H^1_0(\mathcal{O}_n)$ for
all $n$ such
that $u_n \to u$ in $H^1(\CR^d)$.
\end{itemize}

We give two examples of Mosco convergence. We refer to \cite{daners08} for the
proofs and further examples.

\begin{enumerate}
\item If $\mathcal{O}_1\subseteq \mathcal{O}_{2} \subseteq \dots$
and $\bigcup_{n\ge 1}\mathcal{O}_n = \mathcal{O}$, then $\mathcal{O}_n \to
\mathcal{O}$ in the sense of Mosco.
\item If $\mathcal{O}_1 \supseteq \mathcal{O}_2\supseteq\dots$
and $\bigcap_{n\ge 1}\mathcal{O}_n = \mathcal{O}$, then $\mathcal{O}_n \to
\mathcal{O}$ in the sense of Mosco
provided $\mathcal{O}$ satisfies the regularity condition
\[ H^1_0(\mathcal{O}) = \big\{u\in H^1(\CR^d)\,:\, u\equiv 0\,\,\mbox{a.e.\
on}\,\, 
\complement\overline{\mathcal{O}}\big\}.\]
\end{enumerate}

For simplicity we assume that a single set of coefficients $\mathbf{a}$ and
$\mathbf{b}$ is given,
defined on all of $\CR^d$ and satisfying assumptions (i) and (ii) of Theorem \ref{thm:main}.
Given an open set $\mathcal{O}\subseteq\R^d$, we may again consider the form
$\mathfrak{a}$ defined by \eqref{eq.form}.
Associated with this form we obtain, as before, a strongly continuous analytic
semigroup 
$\bS$ on $L^r(\mathcal{O})$. We may consider $\bS$ as a
degenerate 
semigroup on $L^r(\CR^d)$ by extending functions in $L^r(\mathcal{O})$
identically $0$ outside $\mathcal{O}$. Its (not necessarily densely defined) generator
is denoted by $A$.

Suppose next that we are given a sequence of domains $\mathcal{O}_n$ converging to $\mathcal{O}$
in the sense of Mosco, and
consider the corresponding 
forms $\mathfrak{a}_n$ on $H^1_0(\mathcal{O}_n)$. The associated degenerate
semigroups on $L^r(\CR^d)$
are denoted by $\bS_n$ and their (not necessarily densely defined) generators by
 $A_n$.
It is clear from the results of Section \ref{subsect:global} that condition
$\mathrm{(A1')}$  is satisfied. 
We thus have to verify the (pseudo-)resolvent convergence $\mathrm{(A2')}$.
From Theorems 5.2.6 and 4.3.4 of \cite{daners08}, we infer:

\begin{lem}\label{l.mosco} Let $\mathcal{O}$ and $\mathcal{O}_n$ be bounded open
domains contained in some 
fixed bounded subset of $\R^d$ and suppose that $\mathcal{O}_n\to \mathcal{O}$
in the sense of Mosco. 
If {\rm (i)} and {\rm (ii)} of Theorem \ref{thm:main} hold, then for all $1<r<\infty$ we have $A_n \to A$ in the
strong resolvent sense.
\end{lem}

As before we further assume that the functions
$f,g : \CR \to \CR$ satisfy the assumptions (iii) and (v) of Theorem \ref{thm:main}, and 
that the driving processes $W_k$ are independent real-valued standard Brownian
motions. 
We have the following result.

\begin{thm} Let $1<r<\infty$ and $p>\max\{2,r'\}$, and suppose that the initial datum $\xi$
belongs to $L^p(\Omega, \cF_0;L^r(\CR^d))$. Under the above assumptions, let 
$$X:= \mathrm{sol}(A, F, G, \xi), \ \ X_n := \mathrm{sol}(A_n, F, G, \xi),$$
denote the unique mild solutions of the associated stochastic evolution
equations. 
Then $$X_n-S_n(\cdot)\xi \to X -S(\cdot)\xi \ \hbox{in} \  L^q(\Omega; C^\lambda([0,T]; L^r(\CR^d)))$$ 
for all $1\leq q < p$ and $0\leq \lambda < \half -\frac{1}{p}$.
Moreover, 
$  X_n \to X$ in ${L^q(\Omega; C([0,T];L^r(\mathcal{O})))}$
for all $1\le q<p$.
\end{thm}

Once again it is possible to approximate simultaneously the domain
$\mathcal{O}$, the coefficients $\mathbf{a},
\mathbf{b}$, the nonlinearities $f$ and $g$, and the initial datum $\xi$.
We leave the details to the interested reader.

\section*{Acknowledgment}
The authors are grateful to Marta Sanz-Sol\'e for pointing out the references
\cite{bmss95, mss94} and Mark Veraar for a helpful discussion.

\bibliographystyle{siam}
\bibliography{Kunze_vanNeerven_JEE}
\end{document}